\documentclass[a4paper,american,11pt]{amsart}

\usepackage[latin1]{inputenc}
\usepackage[american]{babel}
\usepackage{amsmath}
\usepackage{amssymb}
\usepackage{amsfonts}
\usepackage{tikz-cd}
\usepackage{mathtools}

\usepackage{graphicx}
\usepackage{xcolor}
\usepackage{url}
\usepackage{bbm}
\usepackage{enumitem}
\usepackage[margin = 3cm]{geometry}
\usepackage{setspace}
\usepackage[normalem]{ulem}
\usepackage[colorlinks=true]{hyperref}

\newcommand{\ZZ}{{\mathbb Z}}
\newcommand{\Z}{{\mathbb Z}}

\newcommand{\PP}{{\mathbb P}}

\newcommand{\CC}{{\mathbb C}}
\newcommand{\RR}{{\mathbb R}}

\newcommand{\f}{\Sigma}

\newcommand{\bfv}{\mathbf v}
\newcommand{\bfw}{\mathbf w}
\newcommand{\bfu}{\mathbf u}
\newcommand{\bfe}{\mathbf e}
\newcommand{\bfx}{\mathbf x}

\newcommand{\B}{{\mathcal B}}

\newcommand{\F}{{\mathcal F}}

\newcommand{\A}{{\mathcal A}}
\renewcommand{\L}{{\mathcal L}}

\newcommand{\CH}{A}
\newcommand{\rank}{r}

\newcommand{\comment}[1]{}

\DeclareMathOperator{\Eu}{Eu}
\DeclareMathOperator{\MW}{MW}
\DeclareMathOperator{\bcsm}{CSM}
\DeclareMathOperator{\csm}{csm}

\DeclareMathOperator{\convex}{convex}
\DeclareMathOperator{\cone}{cone}
\DeclareMathOperator{\Vertices}{vert}
\DeclareMathOperator{\spann}{span}

\DeclareMathOperator{\Mat}{Mat}
\DeclareMathOperator{\interior}{int}
\DeclareMathOperator{\face}{face}
\DeclareMathOperator{\faces}{faces}
\DeclareMathOperator{\codim}{codim}
\DeclareMathOperator{\divis}{div}

\DeclareMathOperator{\Bases}{Bases}

\newtheorem{thm}{Theorem}[section]

\newtheorem{prop}[thm]{Proposition}
\newtheorem{lemma}[thm]{Lemma}
\newtheorem{cor}[thm]{Corollary}

\newtheorem{conj}[thm]{Conjecture}

         {\theoremstyle{definition}

\newtheorem{exa*}[thm]{Example}
\newtheorem{defi}[thm]{Definition}
\newtheorem{rem}[thm]{Remark}}

\DeclareRobustCommand{\qedify}[1]{%
  \ifmmode \quad\hbox{#1}
  \else
    \leavevmode\unskip\penalty9999 \hbox{}\nobreak\hfill
    \quad\hbox{#1}%
  \fi
  }
\newenvironment{exa}{\begin{exa*}\pushQED{\qedify{$\diamondsuit$}}}{\popQED\end{exa*}}

\makeatletter
\newtheorem*{rep@theorem}{\rep@title}
\newcommand{\newreptheorem}[2]{%
\newenvironment{rep#1}[1]{%
 \def\rep@title{{\bf #2 \ref{##1}}}%
 \begin{rep@theorem}}%
 {\end{rep@theorem}}}
\makeatother

\newreptheorem{theorem}{Theorem}

\numberwithin{equation}{section}

\begin{document}


\title{Chern-Schwartz-MacPherson cycles of matroids}

\date{\today}

\author{Lucía López de Medrano}
\address{Unidad Cuernavaca del Instituto de Matemáticas, Universidad Nacional
Autónoma de México. Cuernavaca, México.}
\email{lucia.ldm@im.unam.mx}

\author{Felipe Rincón}
\address{Department of Mathematics, University of Oslo, 0851 Oslo, Norway.}
\email{feliperi@math.uio.no}

\author{Kristin Shaw}
\address{Max Planck Institut f\"ur Mathematik in der Naturwissenschaft, Inselstra{\ss}e 22, 04103 
Leipzig, Deutschland}
\email{kshaw@mis.mpg.de}


\begin{abstract}
We define Chern-Schwartz-MacPherson (CSM) cycles of an arbitrary matroid. These are balanced weighted fans supported on the skeleta of the corresponding Bergman fan. 
In the case that the matroid arises from a complex hyperplane arrangement $\mathcal A$, we show that these cycles represent the CSM class of the complement of $\mathcal A$.
We also prove that for any matroid, the degrees of its CSM cycles are given by the coefficients of (a shift of) the reduced characteristic polynomial, 
and that CSM cycles are valuations under matroid polytope subdivisions. 
\end{abstract}

\maketitle

\newcommand{\RP}{\mathbb{RP}}
\newcommand{\CP}{\mathbb{CP}}
\newcommand{\TP}{\mathbb{TP}}
\newcommand{\be}{\beta}
\newcommand{\T}{\mathbb{T}}
\newcommand{\R}{\mathbb{R}}
\newcommand{\C}{\mathbb{C}}

\tableofcontents

\section{Introduction}
Matroids are a combinatorial abstraction of independence in mathematics introduced independently by Whitney and Nakasawa \cite{Nishimura}. 
They axiomatize different notions such as linear independence, algebraic independence, affine independence, and many others.
In particular, every hyperplane arrangement gives rise to a matroid, as we describe in Section \ref{sec:CSMhyperplane}. 
Given an invariant of a hyperplane arrangement, it is thus important to ask if it is an invariant of its underlying matroid. 

In complex algebraic geometry, the Chern-Schwartz-MacPherson class is a generalization of the Chern class of a tangent bundle to the case of singular or non-compact algebraic varieties over $\CC$.
Given a hyperplane arrangement $\A$ in $\CC\PP^d$, its complement $C(\A) \coloneqq \CC\PP^d \setminus \mathcal A$ embeds into the wonderful compactifications, 
as defined by De Concini and Procesi \cite{deConciniProcesi}. 
In this paper we provide a combinatorial description of the Chern-Schwartz-MacPherson class of $C(\A)$ in the maximal wonderful compactification in terms of certain balanced polyhedral fans 
that depend only on the underlying matroid. 
Our combinatorial definition generalizes to all matroids, whether or not they are representable over $\CC$. 

Any matroid $M$ gives rise to a polyhedral fan $\B(M)$ called the Bergman fan of $M$ (Definition \ref{def:BergFan}).
Bergman fans of matroids are fundamental examples of linear spaces in tropical geometry and are thus essential objects in the field.
The Bergman fan $\B(M)$ of a representable matroid $M$ is the tropicalization of any linear space that represents $M$, while
Bergman fans of non-representable matroids are non-realizable tropical varieties \cite{MaclaganSturmfels, BriefIntro}. 
Regardless of whether or not a matroid is representable, the tropical geometry of its Bergman fan is in many ways analogous to the geometry of a classical  non-singular algebraic variety. 
For example, Bergman fans of matroids have a well-behaved intersection ring \cite{ShawInt},
they exhibit a version of Poincar\'e duality for tropical cohomology \cite{JSS}, and their Chow cohomology rings satisfy a version of Hard Lefschetz and the Hodge-Riemann bilinear relations \cite{AdiprasitoHuhKatz}. 
These powerful properties  were used in \cite{AdiprasitoHuhKatz} to resolve Rota's conjecture on the log-concavity of the coefficients of the characteristic polynomial of a general matroid. 

In this paper we define the Chern-Schwartz-MacPherson (CSM) cycles of an arbitrary matroid 
as  tropical cycles supported on the different skeleta of the corresponding Bergman fan. 
This construction is motivated in part by the desire to have a more general theory of characteristic classes in tropical geometry. 
Nonetheless, CSM cycles of matroids are interesting combinatorial objects on their own
 and are useful from a purely matroid-theoretical perspective. 
In fact, the CSM cycles of a matroid can be thought of as balanced polyhedral fans that generalize its Bergman fan to lower dimensions.

The $k$-th CSM cycle of a matroid $M$ is a weighted fan supported on the $k$-dimensional skeleton of the Bergman fan $\B(M)$, with weights coming from the product of beta invariants of certain minors of $M$ (Definition \ref{def:chernweight}). 
The maximal dimensional CSM cycle of $M$ is equal to $\B(M)$ with weights equal to one on all top-dimensional cones,  while the zero dimensional CSM cycle of $M$ is equal to the origin with multiplicity $(-1)^{r(M)-1}\beta(M)$, where $\beta(M)$ is the beta invariant of $M$ and $r(M)$ is the rank of $M$.
The CSM cycles of intermediate dimensions have weights that 
generalize these two cases. 
Our first theorem is that for any $k$, this 
choice of weights on the $k$-skeleton of $\B(M)$ does produce a tropical cycle. 
\begin{reptheorem}{thm:balanced}
The $k$-th CSM cycle $\csm_k(M)$ of a matroid $M$ is a balanced fan.
\end{reptheorem}

Given a complex hyperplane arrangement $\A$ in $\CP^d$, elements in the Chow homology $A_*(W_\A)$ of the maximal wonderful compactification $W_\A$
of the complement $C(\A) \coloneqq \CP^d \setminus \A$ can be represented by balanced fans supported on the Bergman fan $\B(M_\A)$
of the matroid $M_\A$ induced by $\A$, see Section \ref{sec:CSMhyperplane}.
Our second theorem relates the CSM class of the complement $C(\A)$ to the CSM cycles of the matroid $M_ \A$ in the Chow homology of $W$. 
\begin{reptheorem}{thm:csmcomplement}
Let $W_\A$ be the wonderful compactification of 
the complement $C(\A)$ of an arrangement of hyperplanes $\A$ in $\CC\PP^d$. Then 
$$\textstyle \bcsm(\mathbbm{1}_{{C}(\A)}) = \sum_{k=0}^d \csm_k(M_\A) \in A_*(W_\A).$$
\end{reptheorem}

The above theorem shows that the combinatorially defined CSM cycles of a matroid have geometric meaning when the matroid is representable in characteristic $0$. 

For general matroids, we show in Section \ref{Sec:val} that the CSM cycles are matroid valuations.
A matroid valuation is a function on the set of matroids that satisfies an inclusion-exclusion property for matroid polytope subdivisions (Definition \ref{def:valuation}). The class of matroid valuations includes many well-known invariants such as the Tutte polynomial, the volume and Erhart polynomial of the matroid polytope, and the Billera-Jia-Reiner quasisymmetric function \cite{Speyer, BJR, AFR, FinkDerksen}. Matroid valuations have gained significant attention recently and are very useful tools for understanding the combinatorial structure of matroid polytope subdivisions and tropical linear spaces \cite{Speyer, Speyer:Ktheory}.
\begin{reptheorem}{thm:valuation}
For any $k$, the function $\csm_k$ sending a matroid $M$ to its $k$-dimensional CSM cycle $\csm_k(M)$ is a valuation under matroid polytope subdivisions.
\end{reptheorem}

Every tropical cycle in $\mathbb{R}^n$ has a 
degree (Definition \ref{def:degree}). In Section \ref{Sec:polynomial} 
we show that the degrees of the CSM cycles of a matroid are given by the coefficients of a shift of the reduced characteristic polynomial.
These coefficients are of enumerative interest. 
For instance, they provide the $h$-vector of the broken circuit complex of a matroid. 
\begin{reptheorem}{thm:hvector}
Suppose $M$ is a rank $d+1$ matroid. Then
$$\sum_{k=0}^d \deg(\csm_{k}(M))\,t^k = \overline{\chi}_M(1+t).$$
\end{reptheorem}
\noindent For matroids representable in characteristic $0$, the above statement specialises to a formula already found in different contexts (\cite[Theorem 3.5]{Huh:MaxLikely} and \cite[Theorem 1.2]{Aluffi:Grothendieck}).

In  Section \ref{Sec:polynomial}, we state a  conjectural description of Speyer's $g$-polynomial using the CSM cycles of a  matroid.  
This polynomial matroid invariant was originally constructed for matroids representable over a field of characteristic $0$ 
via the $K$-theory of the Grassmannian \cite{Speyer:Ktheory}. 
This definition was later extended to all matroids  in \cite{FinkSpeyer}.
The fact that its coefficients are  non-negative integers for matroids realizable in characteristic $0$ is the key ingredient
in Speyer's proof of the $f$-vector conjecture in characteristic $0$ \cite{Speyer:Ktheory}. 
This conjectured formula describes the $g$-polynomial in terms of intersection numbers of the CSM cycles of a matroid 
with certain tropical cycles derived from them (Conjecture \ref{conj:gpoly}). 
This conjecture provides a Chow theoretic description of this $K$-theoretic invariant. 
A proof of Conjecture \ref{conj:gpoly}  will appear in forthcoming work of Fink, Speyer and the third author.

\subsection*{Previous work}
To end this introduction we would like to point out how the CSM cycles of matroids are related to existing work in tropical geometry. 
In \cite{MIkICM}, Mikhalkin introduced  the tropical canonical class $K_V$ of a tropical variety $V$. 
This is a weighted polyhedral complex supported on the codimension-$1$ skeleton of $V$. The weight in $K_V$ of a codimension-$1$ face $F$ is $\text{val}(F)  - 2$, where $\text{val}(F)$ is equal to the number of top-dimensional faces of $V$ adjacent to $F$. For the Bergman fan $\B(M)$ of a rank $d+1$ matroid $M$ we have $\csm_{d-1}(\B(M)) = -K_{\B(M)}$, see Example \ref{ex:csm0}. 
In the case of tropical curves, this is the same definition of the canonical class used to study tropical linear series and the Riemann-Roch theorem \cite{BakerNorine1, GatKer, MikZha:Jac}.

It is important to notice that for an arbitrary tropical variety $V$, the weighted polyhedral complex $K_V$ is in general not balanced. 
For instance, there is a $2$-dimensional tropical variety 
$V \subseteq \R^4$ presented in \cite[Section 5]{BabaeeHuh}, for which it can be easily checked that $K_V$ does not satisfy the balancing condition.  
This particular tropical variety provides  a  counter-example to the strongly positive Hodge conjecture, and hence is not realizable. 

In general, Mikhalkin also suggested to define the Chern classes of a tropical variety as tropical cycles supported on the skeleta of the variety, however, the weights of these cycles were not defined. 
The definition of the CSM cycles for matroids presented here extends to tropical manifolds, as defined for example in \cite{MikZha:Eig} or  \cite{Shaw:Surf}. These are tropical varieties which are locally given by Bergman fans of matroids. 
In dimension $2$, the canonical class and second Chern classes of (combinatorial) tropical surfaces defined in \cite{Cartwright:Surfaces} and \cite{Shaw:Surf} coincide with $-\csm_1(\B(M))$ and $\csm_0(\B(M))$ respectively, when the tropical surface is the Bergman fan of a rank $3$ matroid $M$. 
These tropical characteristic classes appear in a version of Noether's formula in both of these papers. 

Finally, Bertrand and Bihan equip with weights the skeleta of a complete intersection of tropical hypersurfaces to produce tropical varieties \cite{BertrandBihan}. In Remark \ref{rem:BB}, we address when our constructions overlap  and show that in these cases they coincide.  The connection described in Section \ref{sec:CSMhyperplane} between the CSM cycles of matroids and the CSM class of the complement of a complex hyperplane arrangement suggests there is a relation between the weighted skeleta from \cite{BertrandBihan} and the CSM classes of very affine varieties.

\section*{Acknowledgements}

We are very grateful to Beno\^it Bertrand, Fr{\'e}d{\'e}ric Bihan, Erwan Brugall{\'e}, Gilberto Calvillo, Dustin Cartwright, Alex Fink, Eric Katz, Ragni Piene, and David Speyer for illuminating discussions.
We would also like to thank Erwan Brugall\'e for helpful comments on a preliminary version of this manuscript.

The first author was supported for this research by ECOS NORD M14M03, FORDECYT, UMI 2001, Laboratorio Solomon Lefschetz CNRS-CONACYT-UNAM, México, \linebreak 
PAPIIT-IN114117 and PAPIIT- IN108216. The second author was supported by the Research Council of Norway grant 239968/F20. The research of the third author was supported by the Alexander von Humboldt Foundation. This work was carried out in part while the third author was at the Fields Institute for Research in the Mathematical Sciences for the program ``Combinatorial Algebraic Geometry" and also at the Max Planck Institute for Mathematics in the Sciences.

\section{CSM cycles of Bergman fans of matroids}\label{Sec:CSMcycles}

In this section we define the Chern-Schwartz-MacPherson (CSM) cycles  of a matroid 
as a collection of weighted rational polyhedral fans (Definition \ref{def:chernweight}).
We then prove that these fans are balanced (Theorem \ref{thm:balanced}). 

We start by fixing some notation. 
Throughout we will always consider the standard lattice $\ZZ^{n+1} \subseteq \R^{n+1}$, and we will denote by $\{\bfe_0, \bfe_1, \dotsc, \bfe_n\}$ the standard basis of this lattice. 
For any subset $S \subseteq \{0,\dotsc,n\}$, let $\bfe_S \coloneqq \sum_{i\in S} \bfe_i \in \ZZ^{n+1}$.
The quotient vector space $\R^{n+1} / \mathbf{1} \coloneqq \RR^{n+1} /\,\RR \cdot \bfe_{\{0,\dotsc, n\}}$ is spanned by the lattice $\ZZ^{n+1} / \mathbf{1} \coloneqq \ZZ^{n+1} /\,\ZZ \cdot \bfe_{\{0,\dotsc, n\}}$. 

A polyhedral fan $\f$ in $\R^{n+1}$ is called {\bf rational} if every cone of $\f$ is defined by a collection of  inequalities each of the form $ \langle \alpha ,  \mathbf x  \rangle \leq 0$ with $\alpha  \in \Z^{n+1}$. 
If $\f$ is a rational polyhedral fan in $\RR^{n+1}$ whose lineality space contains $\RR \cdot \bfe_{\{0,\dotsc, n\}}$, we also
refer to its image in $\RR^{n+1}  / \mathbf{1}$ as a rational polyhedral fan.

We will assume the reader has some knowledge of the basics of matroid theory; this can be found, for example, in \cite{White1, White2}.
We denote by $\Mat_{n+1}$ the set of matroids on $n+1$ elements labeled $0,1,\dotsc, n$.
Every matroid has an associated rational polyhedral  fan, called its Bergman fan.
Given a set of vectors $\{ \bfv_1,\dotsc, \bfv_r\}$ in a real vector space, we will denote by $\cone({ \bfv}_1,\dotsc,{  \bfv}_r)\coloneqq \{\sum_{i=1}^r \lambda_i{ \bfv}_i \mid \lambda_i\in\R_{\geq 0}\}$ the cone that they generate.

\begin{defi}\label{def:BergFan}
Let $M\in\Mat_{n+1}$ be a matroid of rank $d+1$. 
If $M$ is a loopless matroid, the {\bf affine Bergman fan} $\hat{\B}(M)$ of $M$ 
is the pure $(d+1)$-dimensional rational polyhedral fan in $\R^{n+1}$ consisting of the collection of 
cones of the form 
\[\sigma_\F \coloneqq  \cone(\bfe_{F_1},\bfe_{F_2},\dotsc,\bfe_{F_k}) + \RR \!\cdot\! \bfe_{\{0,\dotsc, n\}}\]
where 
$\F = \{\emptyset \subsetneq F_1 \subsetneq F_2 \subsetneq \dotsb \subsetneq F_k \subsetneq \{0,\dotsc, n\}\}$ is a chain of flats in 
the lattice of flats $\L(M)$ of $M$. If $M$ has a loop then we define $\hat{\B}(M) = \emptyset$.
 
The {\bf (projective) Bergman fan} $\B(M)$ of $M$
is the pure $d$-dimensional rational polyhedral fan obtained as the image of $\hat{\B}(M)$ 
in the quotient vector space $\R^{n+1} / \mathbf{1}$.
\end{defi}

\begin{exa}\label{ex:bergmanuniform}
Suppose $M$ is the uniform matroid $M = U_{d+1,n+1}$. 
A subset $F \subseteq \{0,\dotsc,n\}$ is a flat of $M$ if and only if $|F| \leq d$ or $F = \{0, \dotsc, n\}$. 
The top-dimensional cones of the Bergman fan $\B(M)$ are thus all cones of the form
$\cone(\bfe_{F_1},\bfe_{F_2},\dotsc,\bfe_{F_d})$ 
with $F_1 \subsetneq \dotsb \subsetneq F_d \subsetneq  \{0,\dotsc,n\}$ and $|F_i|=i$.
Figure \ref{fig:tropicalplane} shows the $2$-dimensional Bergman fan $\B(U_{3,4})$ in $\R^4 / \mathbf{1}  \cong \RR^3$.
\begin{figure}[ht]
\begin{center}
\includegraphics[scale=1.5]{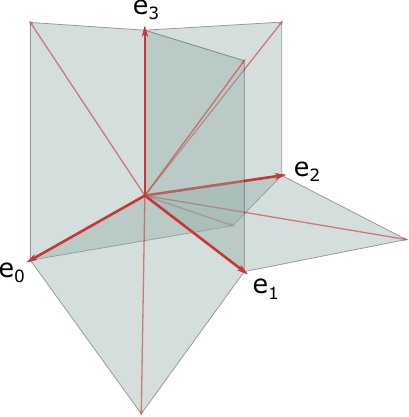}
\caption{The Bergman fan $\B(U_{3,4})$ in $\R^4 / \mathbf{1} \cong \RR^3$.}
\label{fig:tropicalplane}
\end{center}
\end{figure}
\end{exa}

A pure dimensional polyhedral fan $\f$ is {\bf weighted} if each top-dimensional cone $\sigma \in \f$ is equipped with an integer weight $w_{\f}(\sigma) \in \ZZ$. 
For a polyhedral fan $\f$,  we write $|\f| \coloneqq  \bigcup_{\sigma \in \f} \sigma$ for its {\bf support}. 
If $\f$ is a weighted polyhedral fan, we define its support to be the union of its top-dimensional cones of non-zero weight. 

Let $\f$ be a pure $d$-dimensional rational weighted polyhedral fan in $\RR^{n+1} / \mathbf{1}$. Suppose $\tau \in \f$ is a $(d-1)$-dimensional cone, 
and consider the linear subspace $L_\tau \coloneqq \spann_\RR(\tau)$. 
For any $d$-dimensional cone $\sigma \in \f$ such that $\sigma \supsetneq \tau$, let $\mathbf v_\sigma \in \ZZ^{n+1} / \mathbf{1}$ be such that 
$$\spann_\ZZ ( \mathbf v_\sigma, L_\tau \cap (\ZZ^{n+1} / \mathbf{1})) = \spann_\RR(\sigma) \cap (\ZZ^{n+1} / \mathbf{1}).$$
The fan $\f$ satisfies the {\bf balancing condition} at $\tau$ if 
$\sum_{\sigma \supsetneq \tau} w_{\f}(\sigma) \mathbf v_{\sigma} \in L_\tau$.
We say that $\f$ is {\bf balanced} if every $(d-1)$-dimensional cone $\tau$ of $\f$ verifies the balancing condition.

\begin{prop}\cite{SturmPoly}
The Bergman fan of a matroid is a balanced 
fan when equipped with weights equal to 1 on all its top-dimensional cones. 
\end{prop}

We will often not be concerned with the specific fan structure of a weighted polyhedral fan, but only with its support and its weights. This prompts us to introduce the notion of fan tropical cycles. 

\begin{defi}\label{def:tropCycle}
A {\bf fan tropical cycle} in $\R^{n+1} / \mathbf{1}$ is a pure dimensional balanced rational 
weighted fan in $\R^{n+1} / \mathbf{1}$ up to an equivalence relation. 
Given two such fans $\f$ and $\f'$, we have $\f \sim \f'$ if $|\f| = |\f'|$ and 
whenever $\sigma \in \f$ and $\sigma' \in \f'$ are top-dimensional cones such that $\interior(\sigma) \cap \interior(\sigma') \neq \emptyset$ 
we have $w_{\f}(\sigma) = w_{\f'}(\sigma')$.

The set of $k$-dimensional fan tropical cycles in $\R^{n+1} / \mathbf{1}$  is denoted by 
$\mathcal{Z}_k(\R^{n+1} / \mathbf{1})$. 
This set forms a group under the operation of taking set theoretic unions along with the  addition of weight functions \cite[Construction 2.13]{AllermannRau}.
\end{defi}

\begin{defi}\label{def:matCycle} The {\bf matroidal tropical cycle} associated to a matroid $M\in\Mat_{n+1}$ is the tropical cycle 
represented by the polyhedral fan $\B(M)$ equipped with weights equal to 1 on all its top-dimensional cones. 
\end{defi}

We will use the notation $\B(M)$ to denote both the Bergman fan of a matroid $M$ and the tropical cycle it defines.

We will define the CSM cycles of a matroid $M$ by assigning a natural weight to each cone of the Bergman fan $\mathcal{B}(M)$. 
The main ingredient to concoct these weights is the beta invariant of a matroid. 

\begin{defi}\label{def:mobiusBeta}
Let $\L(M)$ be the lattice of flats of a matroid $M$. 
The {\bf M\"obius function} of $\L(M)$ is the function $\mu : \L(M) \times \L(M) \to \ZZ$ defined recursively by
\[ \mu(F,G) \coloneqq  
\begin{cases} 
0 & \text{ if } F \nsubseteq G,\\
1 & \text{ if } F = G,\\
\displaystyle - \sum_{F \subseteq G' \subsetneq G} \mu(F,G') & \text{ if } F \subsetneq G. 
\end{cases}
\] 

Let $r$ denote the rank function of $M$. 
If $M$ is a loopless matroid, the {\bf characteristic polynomial} $\chi_M(\lambda)$ of $M$ is the polynomial
\[ \chi_M(\lambda) \coloneqq  \sum_{F \in \L(M)} \mu(\emptyset, F) \, \lambda^{r(M)-r(F)}. \]
If $M$ has a loop, we define $\chi_M(\lambda) \equiv 0$. 
The {\bf reduced characteristic polynomial} of $M$ is the polynomial 
$$\overline{\chi}_M(\lambda) \coloneqq  \chi_M(\lambda)/(\lambda - 1).$$ 

If $M$ is a loopless matroid, the {\bf beta invariant} of $M$ is defined as
\begin{equation*}\label{eq:betainvariant}
\beta(M) \coloneqq  (-1)^{r(M)} \sum_{F \in \L(M)} \mu(\emptyset, F) \, r(F) = (-1)^{r(M)-1}\overline{\chi}_M(1).
\end{equation*}
If $M$ has a loop then $\beta(M)$ is defined to be $0$.  
The beta invariant of a matroid is always non-negative, and furthermore, 
$\beta(M) = 0$ if and only if $M$ is disconnected or $M$ 
consists of a single loop. 
For a more detailed exposition of these notions, see, for instance, \cite[Chapter 7]{White2}.
\end{defi}

\begin{exa}\label{ex:betauniform}
Consider the uniform matroid $M = U_{d+1, n+1}$, discussed in Example \ref{ex:bergmanuniform}. 
Its M\"obius function satisfies $\mu(\emptyset,F) = (-1)^{|F|}$ if $|F| \leq d$, 
and $\mu(\emptyset,\{0,\dotsc,n\}) = \sum_{i=0}^{d}(-1)^{i+1}\binom{n+1}{i}$.
Its characteristic polynomial is thus equal to
\[
\chi_{U_{d+1,n+1}}(\lambda) = \sum_{i=0}^{d} (-1)^i \binom{n+1}{i} (\lambda^{d+1-i}-1),
\]
and its reduced characteristic polynomial is equal to
\[
\overline{\chi}_{U_{d+1,n+1}}(\lambda)
=  \sum_{i=0}^{d} (-1)^i \binom{n}{i} \lambda^{d-i}.
\]
The beta invariant of $U_{d+1, n+1}$ is $\beta(U_{d+1,n+1}) = \binom{n-1}{d}.$
\end{exa}

The following is the central definition of our paper.

\begin{defi}\label{def:chernweight}
Suppose $M\in\Mat_{n+1}$ is a  rank $d+1$ matroid.
For $0 \leq k \leq d$, the $k$-dimensional {\bf Chern-Schwartz-MacPherson (CSM) cycle}
$\csm_k(M)$ of $M$ is the $k$-dimensional skeleton of $\B(M)$ equipped with weights on its top-dimensional cones. 
If $M$ is a loopless matroid,  the  
weight of the cone $\sigma_\F$ corresponding to a flag of flats $\F \coloneqq  \{\emptyset = F_0
\subsetneq F_1 \subsetneq \dotsb \subsetneq F_{k} \subsetneq F_{k+1} =
\{0,\dotsc,n\}\}$ is 
\[ w(\sigma_\F) \coloneqq  (-1)^{d-k} \prod_{i=0}^{k} \beta(M|F_{i+1}/F_i),\]
where $M|F_{i+1}/F_i$ denotes the minor of $M$ obtained by restricting to $F_{i+1}$ and contracting $F_i$.
If $M$ has a loop then we define $\csm_k(M) \coloneqq  \emptyset$ for all $k$.
\end{defi}

We will prove in Theorem \ref{thm:balanced} that the CSM cycles of a matroid are balanced fans.
As the name suggests, we will often consider the CSM cycles of a matroid as fan tropical cycles, as in Definition \ref{def:tropCycle}.
We illustrate our definition with some examples.

\begin{exa}\label{ex:csm0}
The $0$-dimensional CSM cycle of any rank $d+1$ 
matroid  $M \in \Mat_{n+1}$ is the origin in $\R^{n+1} / \mathbf{1}$ 
with weight equal to $(-1)^d\beta(M)$. 
For example, the cycle $\csm_0(U_{2,3})$ consists of the origin
with weight $-1$, while $\csm_0(U_{2,4})$ is equal to the origin with weight $-2$; see Figure \ref{fig:csm0}.
\begin{figure}
\begin{center}
\includegraphics[scale=1.3]{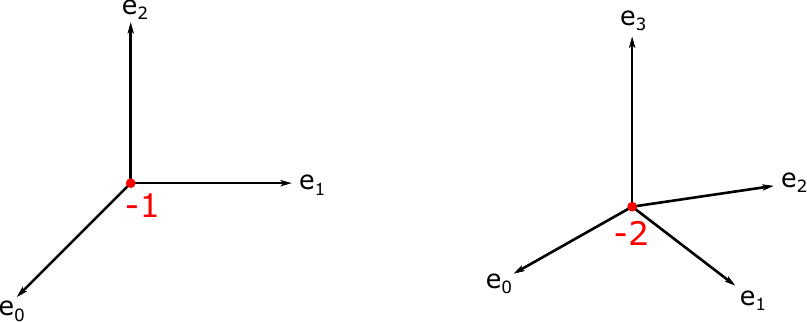}
\caption{The CSM cycles $\csm_0(U_{2,3})$ in $\B(U_{2,3})$ in $\RR^3 / \mathbf{1} \cong \R^2$ and $\csm_0(U_{2,4})$ in $ \B(U_{2,4})$ in $\RR^4 / \mathbf{1} \cong \R^3$. }
\label{fig:csm0}
\end{center}
\end{figure}

The $d$-dimensional CSM cycle of a  matroid $M$ is equal to the matroidal tropical cycle $\B(M)$.
Indeed, if $M$ is a rank $d+1$ loopless matroid and $\{\emptyset = F_0 \subsetneq F_1 \subsetneq \dotsb \subsetneq F_{d} \subsetneq F_{d+1} = \{0,\dotsc,n\}\}$
is a maximal chain of flats then all the matroids $M|F_{i+1}/F_i$ are uniform matroids of rank 1, which have beta invariant equal to 1.

The $(d-1)$-dimensional CSM cycle of a matroid $M$ consists of the codimension-$1$ skeleton of $\B(M)$ with certain weights. The weight of a $(d-1)$-dimensional face $\sigma_{\F}$ is given by $2 - \text{val}(\sigma_{\F})$, where $\text{val}(\sigma_{\F})$ is the number of top-dimensional faces of $\B(M)$ containing $\sigma_{\F}$. 
This is because for any length $d$ chain $\F = \{\emptyset = F_0 \subsetneq F_1 \subsetneq \dotsb \subsetneq F_{d-1} \subsetneq F_{d} = \{0,\dotsc,n\}\}$ there is a unique $j$ for which $r(F_{j+1}) = r(F_j)+2$. For all $i \neq j$ the matroids  $M|F_{i+1}/F_i$ are uniform matroids of rank 1 as above. The  matroid $M|F_{j+1}/F_j$ is of rank $2$ and has beta invariant $\beta(M) = \text{val}(\sigma_{\F}) - 2$.  
\end{exa}

\begin{exa}\label{ex:csmcycleUniform}
Let us consider again the case of the uniform matroid $M = U_{d+1, n+1}$, with $d <n$ (see Examples \ref{ex:bergmanuniform} and \ref{ex:betauniform}). 
If $\F \coloneqq  \{\emptyset = F_0 \subsetneq F_1 \subsetneq \dotsb \subsetneq F_{k} \subsetneq F_{k+1} 
= \{0,\dotsc,n\}\}$ is a chain of flats in $M$, all the matroids $M|F_{i+1}/F_i$ with $i < k$ are  direct sums of coloops. 
Therefore,  $\beta(M|F_{i+1}/F_i) = 1$ if $\rank(M|F_{i+1}/F_i) =1$ and $\beta(M|F_{i+1}/F_i) =0$ otherwise. 
The weight of the cone $\sigma_\F$ in $\csm_k(M)$ is thus zero unless $|F_i| = i$ for all $i \leq k$. 
Equivalently, the cone  $\sigma_\F$ is equipped with weight $0$ unless it is a top-dimensional cone of the Bergman fan $\B(U_{k+1,n+1})$.
In this case, the matroid $M|F_{k+1}/F_k = M/F_k$ is a uniform matroid of rank $d+1-k$
on $n+1-k$ elements. By Example \ref{ex:betauniform}, its  beta invariant is $\beta(U_{d+1-k, n+1-k}) = \binom{n-k-1}{d-k}$. 
It follows that, as a tropical cycle, the CSM cycle $\csm_k(U_{d+1, n+1})$ is 
the Bergman fan $\B(U_{k+1,n+1})$ equipped with weight $(-1)^{d-k}\binom{n-k-1}{d-k}$ on all its top-dimensional cones.
For example, the CSM cycle $\csm_1(U_{3, 4})$ consists of the rays in directions $\mathbf{e}_0, \mathbf{e}_1, \mathbf{e}_2,$ and $\mathbf{e}_3$, each equipped with weight $-1$; see Figure \ref{fig:tropicalplane}.
\end{exa}

Notice that some of the cones in the $k$-skeleton of $\B(M)$ can be assigned weight $0$ in $\csm_k(M)$. 
The following proposition describes the support of $\csm_k(M)$ in terms of the coarse subdivision of $|\B(M)|$, introduced in \cite{ArdilaKlivans}.
Cones of this coarse subdivision correspond to equivalence classes of cones in $\B(M)$. Two cones $\sigma_{\F}$ and $\sigma_{\F'}$ associated to chains of flats
$\F \coloneqq  \{\emptyset = F_0 \subsetneq F_1 \subsetneq \dotsb \subsetneq F_{k} \subsetneq F_{k+1} = \{0,\dotsc,n\}\}$ and $\F' \coloneqq  \{\emptyset = F'_0 \subsetneq F'_1 \subsetneq \dotsb \subsetneq F'_{l'} \subsetneq F'_{l'+1} = \{0,\dotsc,n\}\}$
are equivalent if and only if the matroids
$$M|F_1/F_0 \oplus M|F_2/F_1 \oplus \dotsb \oplus M|F_{k+1}/F_k \,\text{ and }\,  M|F'_1/F'_0 \oplus M|F'_2/F'_1 \oplus \dotsb \oplus M|F'_{l'+1}/F'_{l'}$$
are equal. 
Such an equivalence class of cones of $\B(M)$ produces an $m$-dimensional cone in the coarse subdivision, where $m$ is the number of connected components of the matroids described above.

\begin{exa}\label{ex:coarse}
For  $M = U_{d+1, n+1}$, the Bergman fan $\B(M)$ is described in Example \ref{ex:bergmanuniform}. The coarse subdivision of  $\B(M)$ has as top-dimensional cones all cones of the form 
$\cone(\bfe_i \mid i \in I)$ with $I \subseteq \{0,\dotsc,n\}$ and $|I| = d$.
\end{exa}

\begin{prop}\label{prop:nonzerocones}
The support of $\csm_k(M)$ is equal to the $k$-skeleton of the coarse subdivision of $|\B(M)|$.  
\end{prop}
\begin{proof}
If $M$ is a loopless matroid, the weight of the cone $\sigma_\F$ corresponding to a flag of flats
$\F \coloneqq  \{\emptyset = F_0 \subsetneq F_1 \subsetneq \dotsb \subsetneq F_{k} \subsetneq F_{k+1} = \{0,\dotsc,n\}\}$ 
is non-zero precisely when all the loopless matroids $M|F_{i+1}/F_i$ are connected. 
This happens precisely when $\sigma_\F$ is contained in a $k$-dimensional cone of the coarse subdivision of $|\B(M)|$.
\end{proof}

\begin{exa}
A matroid is called {\bf series-parallel} if it is the matroid associated to a series-parallel network;
see \cite[Section 5.4]{Oxley}. Equivalently, a matroid $M$ is series-parallel if and only if $\beta(M) = 1$
or $M$ is a loop \cite[Theorem 7.3.4]{White2}. 
Furthermore, any minor of a series-parallel matroid is either disconnected 
or a series-parallel matroid \cite[Corollary 5.4.12]{Oxley}. 
It follows that if $M$ is a rank $d+1$ series-parallel matroid then for any $k$, the weights on the top-dimensional cones of 
$\csm_k(M)$ are all either $0$ or $(-1)^{d - k}$. 
In view of Proposition \ref{prop:nonzerocones} we conclude that, as a tropical cycle,
the CSM cycle $\csm_k(M)$ is equal to the $k$-skeleton of the coarse subdivision of $|\B(M)|$ with all weights equal to $(-1)^{d - k}$.
\end{exa}

The main theorem in this section shows that CSM cycles are balanced fans. 

\begin{thm}\label{thm:balanced}
The CSM cycle $\csm_k(M)$ of a matroid $M$ is a balanced fan.
\end{thm}

The proof follows from the case  $k=1$, which we prove  in the next lemma. 
 
\begin{lemma}\label{lem:balancing1dim}
The CSM cycle $\csm_1(M)$ of a matroid $M$ is a balanced fan.
\end{lemma}
\begin{proof}
Let $M \in \Mat_{n+1}$ be a rank $d+1$ matroid. We can assume that $M$ has no loops, as otherwise $\B(M)$ is empty.
The only codimension-$1$ cone of $\csm_1(M)$ is the origin of $\R^{n+1} / \mathbf{1}$. 
The top-dimensional cones of $\csm_1(M)$ are the cones $\sigma_F = \cone(\bfe_{F})$ 
with $F$ a flat in $\hat{\L}(M) \coloneqq \L(M) \backslash \{\emptyset, \{0,\dots,n\}\}$.
To show that $\csm_{1}(M)$ is balanced at the origin, we must show that 
$\textstyle \sum w(\sigma_F)  \bfe_F  = 0$ in $\R^{n+1} / \mathbf{1}$,  where the sum is over all flats $F \in \hat{\L}(M)$. 
This is equivalent to
\begin{equation}\label{eq:balanced}
\sum_{F \in \hat{\L}(M)} \beta(M|F) \beta(M/F) \, \bfe_F \, \in \, \RR \!\cdot\! \bfe_{\{0,\dotsc,n\}}
\end{equation} 
in the vector space $\R^{n+1}.$ 
For any $0 \leq i \leq n$, the $i$-th coordinate of the sum in \eqref{eq:balanced} is 
\begin{equation}\label{eq:coordinate}
\sum_{F \ni i} \beta(M|F) \beta(M/F),
\end{equation}
where the sum if over all flats $F \in \hat{\L}(M)$ containing $i$.
To prove \eqref{eq:balanced} we must then show that the sum in \eqref{eq:coordinate} 
is independent of the choice of $i \in \{0, \dots, n\}$. 
 
For any $F \in \hat{\L}(M)$, the lattice of flats of the matroid $M|F$ is isomorphic to the interval
$[\emptyset, F]$ of $\L(M)$, and the lattice of flats of $M/F$ is isomorphic
to the interval $[F, \{0,\dotsc, n\}]$ of $\L(M)$. 
These intervals correspond to loopless matroids, so the sum in \eqref{eq:coordinate} is equal to
\begin{align*}
& = \sum_{F \ni i} \Bigl( (-1)^{r(F)}
\sum_{F_1 \subseteq F} \mu(\emptyset, F_1) r(F_1) \Bigr) \Bigl( (-1)^{d+1-r(F)} \sum_{F_2
\supseteq F} \mu(F, F_2) (r(F_2) -r(F)) \Bigr) \\ 
& = (-1)^{d+1} \sum_{F \ni i} \, \Bigl( \sum_{F_1 \subseteq F} \mu(\emptyset,
F_1) r(F_1) \Bigr) \Bigl( \sum_{F_2 \supseteq F} \mu(F, F_2) r(F_2) \Bigr),
\end{align*}
where the last equality follows from the fact that $\sum_{F_2 \supseteq F}\mu(F,F_2)=0$.
We can now let $F$ vary over all flats of $\L(M)$ that contain $i$ including $\{0,\dotsc,n\}$,
as this just adds the constant term $\beta(M)(d+1)$, which does not depend on $i$.
Reordering the terms in the summation we get
\begin{align*}
& (-1)^{d+1} \sum_{F \ni i} \, \sum_{F_1 \subseteq F, \, F_2 \supseteq F}
\mu(\emptyset, F_1) \mu(F, F_2) r(F_1) r(F_2) \\
& = (-1)^{d+1} \sum_{F_1
\subseteq F_2} \mu(\emptyset, F_1) r(F_1) r(F_2) \Bigl( \sum_{\substack{F_1
\subseteq F \subseteq F_2 \\ F \ni i}} \mu(F, F_2) \Bigr).
\end{align*}
The condition that  $F_1 \subseteq F$ and $F \ni i$ is  equivalent
to $\overline{F_1 \cup \{i\}}\subseteq F$, where  $\overline{F_1 \cup \{i\}}$ denotes the minimal flat of $M$ containing $F_1 \cup \{i\}$. We can then rewrite the last sum as
\begin{align*}
&= (-1)^{d+1} \sum_{F_1 \subseteq F_2} \mu(\emptyset, F_1) r(F_1) r(F_2)
\Bigl( \sum_{\overline{F_1 \cup \{i\}}\subseteq F \subseteq F_2} \mu(F, F_2)
\Bigr).  
\end{align*} 
If $\bf P$ is a poset with minimum element $\hat 0$, maximum element $\hat 1$, and $\hat 0 \neq \hat 1$,
the M\"obius function $\mu_{\bf P}$ satisfies
$\sum_{p \in \bf P} \mu_{\bf P}(p, \hat 1) = 0$ \cite[Proposition 3.7.2]{StanleyEC1}.
The very last sum in parenthesis is then equal to $0$ whenever the
interval $[\overline{F_1 \cup \{i\}}, F_2]$ of $\L(M)$ is empty or has more
than one element, and it is equal to $1$ when $\overline{F_1 \cup \{i\}} =
F_2$. The above sum is thus equal to
\begin{align}
&= (-1)^{d+1} \sum_{F_1} \mu(\emptyset,
F_1) r(F_1) r(\overline{F_1 \cup \{i\}}) \nonumber \\ &= (-1)^{d+1} \Bigl(
\sum_{F_1 \ni i} \mu(\emptyset, F_1) r(F_1)r(F_1) + \sum_{F_1 \not\owns i}
\mu(\emptyset, F_1) r(F_1) (r(F_1)+1) \Bigr) \nonumber \\
&= (-1)^{d+1} \Bigl( \sum_{F_1} \mu(\emptyset, F_1) r(F_1)^2 + \sum_{F_1
\not\owns i} \mu(\emptyset, F_1) r(F_1) \Bigr) \label{eq:coordsum}.
\end{align}
The polynomial $p_i(\lambda) \coloneqq  \sum_{F_1 \not\owns i}
\mu(\emptyset, F_1) \lambda^{d-r(F_1)}$ does not depend on $i$; in fact, 
if $M$ has no loops then $p_i(\lambda) = \overline{\chi}(\lambda)$ for any $0 \leq i \leq n$ (see
\cite[Corollary 7.2.7]{White2}). 
If $p'_i(\lambda)$ denotes the derivative of $p_i(\lambda)$, we can write 
\[ \sum_{F_1 \not\owns i} \mu(\emptyset, F_1) r(F_1) = -p'_i(1) + d \sum_{F_1
\not\owns i} \mu(\emptyset, F_1) = -p'_i(1) + d \, p_i(1). 
\]
This shows that \eqref{eq:coordsum}, and thus \eqref{eq:coordinate}, does not depend on $i$,
proving \eqref{eq:balanced}.
\end{proof}

\begin{proof}[Proof of Theorem \ref{thm:balanced}]
Let $M \in \text{Mat}_{n+1}$ be a loopless matroid of rank $d+1$. 
 The balancing property of $\csm_k(M)$ for general $k$ will follow from the case $k
 =1$, which was proved in Lemma \ref{lem:balancing1dim}. 
 Let $\tau_\F = \cone( \bfe_{F_1}, \dotsc, \bfe_{F_{k-1}})$ 
 be a $(k-1)$-dimensional cone of $\B(M)$, corresponding to
 the chain of flats $\F \coloneqq  \{\emptyset = F_0 \subsetneq F_1 \subsetneq \dotsb
 \subsetneq F_{k-1} \subsetneq F_{k} = \{0,\dotsc,n\}\}$. The
 $k$-dimensional cones 
 adjacent to $\tau_\F$ have the form $\sigma_G \coloneqq  \tau_\F + \RR_{\geq 0}
 \cdot \bfe_G$, where $G$ is a flat of $M$ 
 such
 that $\F \cup \{G\}$ forms 
 a chain of flats of length $k+1$. Such a flat $G$ sits in
 exactly one of the open intervals $(F_i, F_{i+1})$ of $\L(M)$, with $0 \leq i \leq
 k-1$. Denoting $M_i \coloneqq  M|F_{i+1}/F_i$ and $\beta_i \coloneqq  \prod_{j\neq i} \beta(M_j)$, 
in  $ \RR^{n+1}$ we have
\begin{align*}
  \sum_{G} w(\sigma_G) \, \bfe_G &= \sum_{i=0}^{k-1} \sum_{G \in (F_i, F_{i+1})} w(\sigma_G) \, \bfe_G \\ 
	&= \sum_{i=0}^{k-1} \sum_{G \in (F_i, F_{i+1})} (-1)^{d-k}
  \beta_i \, \beta(M|G/F_i) \, \beta(M|F_{i+1}/G) \, \bfe_G \\
  &= (-1)^{d-k} \sum_{i=0}^{k-1} \beta_i \, \Bigl( \sum_{G \in (F_i, F_{i+1})} \beta(M|G/F_i) \, \beta(M|F_{i+1}/G) \, \bfe_G \Bigr).
 \end{align*}
 The lattice of flats of $M_i$ is isomorphic to the interval $[F_i, F_{i+1}]$ of $\L(M)$, so the last expression is equal to
 \begin{align*}
  \sum_{G} w(\sigma_G) \, \bfe_G &= (-1)^{d-k} \sum_{i=0}^{k-1} \beta_i \, 
	\Bigl( \sum_{G' \in \hat{\L}(M_i)} \beta(M_i|G') \, \beta(M_i/G') \, (\bfe_{G'}+ \bfe_{F_i}) \Bigr).
 \end{align*}
 By the balancing condition in the case $k=1$ (Statement \eqref{eq:balanced}), 
 the very last sum in parenthesis is a vector in the span of $\bfe_{F_{i+1}}$ and $\bfe_{F_i}$. 
 This shows that the whole sum $\sum_{G} w(\sigma_G) \bfe_G$ is a linear combination of $\bfe_{F_0}, \bfe_{F_1}, \dotsc, \bfe_{F_{k}}$, 
 which means that it is in $\spann(\tau_\F) \subseteq  \RR^{n+1}$. This proves that $\csm_k(M)$ is balanced. 
\end{proof}

\section{CSM classes of complements of hyperplane arrangements}\label{sec:CSMhyperplane}
The goal of this section is to relate the CSM class of the complement of a hyperplane arrangement in $\CC \PP^d$ to the CSM cycles of the underlying matroid of the arrangement. 

Let $X$ be an algebraic variety over the complex numbers. The group of {\bf constructible functions} on $X$, denoted by $\mathcal{C}(X)$, 
is the additive group generated by the functions of the form 
$$\mathbbm{1}_Y(x)  \coloneqq 
\begin{cases}
 1 & \text{ if } x \in Y,   \\
 0 & \text{ if } x \notin Y;
\end{cases}
$$
where $Y$ is a subvariety of $X$. 
Let $\mathcal{C}$ be the functor of constructible functions from the category of complex algebraic varieties with proper morphisms to the category of abelian groups.
Let $A_*$ denote the functor from the category of complex algebraic varieties to the category of abelian groups 
assigning to a variety $X$ its Chow homology group $A_*(X)$. 

The {\bf Chern-Schwartz-MacPherson class} 
is the unique natural transformation CSM from $\mathcal{C}$ to $A_*$ such that, if $X$ is smooth and complete then
$$\bcsm(\mathbbm{1}_X)
= c(TX) \cap [X],$$
where $TX$ is the tangent bundle of $X$ and $c(T X)$ denotes its Chern class.
We refer the reader to \cite{Aluffi:characteristic} for an introduction to these and other characteristic classes, as well as to \cite{brasselet} for an account of the interesting history of their development. 

There are two important features of the CSM class.  Firstly, the
dimension zero part of the $\bcsm$ class of a variety gives the topological Euler characteristic:
$$\deg \bcsm_0(\mathbbm{1}_X) = \Eu(X).$$
Secondly, they satisfy an inclusion-exclusion property. Namely, for subvarieties $Y_1, Y_2 \subseteq X$, the CSM class satisfies
\begin{equation*}
\bcsm(\mathbbm{1}_{Y_1 \cup Y_2}) = \bcsm(\mathbbm{1}_{Y_1}) + \bcsm(\mathbbm{1}_{Y_2}) - \bcsm(\mathbbm{1}_{Y_1\cap Y_2}) \in A_*(X).
\end{equation*}

Given an arrangement $\A = \{H_0, \dotsc, H_n\}$ of $n+1$ hyperplanes in $\CP^d$, let
$C(\A) \coloneqq \CC\PP^d \setminus \bigcup_{i=0}^n H_i$ denote its complement.
We will always assume that the hyperplane arrangement $\A$ is {\bf essential}, meaning that $\bigcap_{i=0}^n H_i = \emptyset$.  
The arrangement $\A$ defines a rank $d+1$ matroid $M_{\A}\in\Mat_{n+1}$
with rank function $\rank \colon 2^{\{0, \dots , n\}} \to \mathbb{Z}_{\geq 0}$ given by 
$$\rank(I) = \codim_{\CC} \bigcap_{i \in I} H_i.$$
The flats of $M_\A$ are in one to one correspondence with the linear subspaces of $\CP^{d}$ obtained as intersections of some of the hyperplanes in $\A$. 
Note that we consider $\CP^{d}$ and $\emptyset$ to be two such subspaces, corresponding to the flats $\emptyset$ and $\{0,\dots, n\}$, respectively. 
Indeed, any linear subspace $L$ of $\CP^{d}$ that occurs as the intersection of hyperplanes in $\A$ has the form $L = H_F \coloneqq \bigcap_{i \in F} H_i$, where $F$ is the flat $F = \{i \mid L \subseteq  H_i\}$.
The collection of linear subspaces $H_F$ for $F \in \L(M_{\A})$ ordered by reverse inclusion is a lattice isomorphic to the lattice of flats $\mathcal{L}_\A \coloneqq \L(M_{\mathcal{A}})$. 

Given a hyperplane arrangement $\A$ in $\CC\PP^d$, 
we denote by $W_{\A}$ the  {\bf maximal wonderful compactification} of its complement $C(\A)$, introduced by
De Concini and Procesi \cite{deConciniProcesi}. 
For an introduction to this compactification 
and others from a discrete or tropical-geometric point of view 
see \cite{feichtnerConcini, denhamToric}. 
The maximal  wonderful compactification $W_{\A}$ is obtained from $\CC\PP^d$ 
by blowing up all 
linear subspaces $H_F \subseteq \CP^d$
corresponding to flats $F \in \hat{\L}_\A \coloneqq \L_\A \setminus \{\emptyset, \{0,\dotsc,n\}\}$,
in order of increasing dimension. 
The divisor $\mathcal{D}\coloneqq W_\A \setminus C(\A)$ is a
simple normal crossing divisor, whose irreducible components 
are the proper transforms of the linear subspaces $H_F$. For any $F \in \hat{\L}_{\A}$, we denote the proper transform of $H_F$ in $W_{\A}$ by $D_F$. 

The Chow cohomology ring of the maximal wonderful compactification has
a simple combinatorial description. 
Consider the polynomial ring $S \coloneqq \mathbb{Q}[x_F \mid F \in  \hat \L_{\A}]$,
and define the ideal $I$ generated by 
$$\textstyle \sum_{F \ni i} x_F - \sum_{G \ni j} x_G \text{ for all } i \neq j \quad
\text{ and } \quad  x_Fx_G \text{ if } F \nsubseteq G  \text{ and } G \nsubseteq F.$$
Then the Chow cohomology ring of the maximal wonderful compactification 
is isomorphic to the graded quotient ring $A^*(W_\A) \cong S/I$. 
In this presentation, the variable $x_F$ represents the Chow cohomology class Poincar\'e dual to the class of the divisor  $D_F$. 
The ring $A^*(W_\A)$ is generated by the monomials of the form 
$x_{F_1} \dotsb x_{F_k}$, where $F_1 \subsetneq \dotsb \subsetneq F_k$ 
is a chain of  flats in $\hat \L_{\A}$ \cite[Proposition 5.5]{AdiprasitoHuhKatz}.

The Chow homology groups of the maximal wonderful compactification can be described in polyhedral terms.
A $k$-dimensional {\bf Minkowski weight} of a $d$-dimensional rational  fan $\Sigma$ is a rational weighted balanced 
fan whose underlying fan is equal to the $k$-dimensional skeleton of $\Sigma$.
The sum of two $k$-dimensional Minkowski weights of $\Sigma$ is the Minkowski weight 
obtained by adding the weights cone by cone. 
We denote the group of $k$-dimensional Minkowski weights of $\Sigma$ by $\MW_k(\Sigma)$.

Let $A_k(W_ \A)$ denote the $k$-th Chow homology group of 
$W_{\A}$. Then for every $k$ there is an isomorphism   
\begin{equation}\label{equation:chw}
 A_k(W_\A) \cong \MW_k(\B_\A),
\end{equation}
where $\B_\A$ denotes the Bergman fan of the matroid $M_\A$.
This isomorphism 
is obtained from Kronecker duality 
$\CH_k(W_{\A}) \cong \text{Hom}(\CH^k(W_\A), \mathbb{Z})$ and the perfect pairing defined by
\begin{equation}\label{eqn:compPairingExplicit}
\begin{array}{ccccc}
\CH^k(W_\A) & \times & \MW_{k}(\B_\A) & \longrightarrow &  \Z \\
x_{F_1} \dotsb x_{F_k} & \cap & Z & \longmapsto & w_Z(\sigma_{\mathcal{F}}), 
\end{array}
\end{equation}
where  $F_1 \subsetneq \dotsb \subsetneq F_k$ is a chain of flats in $\hat \L_{\A}$  \cite[Proposition 5.6]{AdiprasitoHuhKatz}.

Using this machinery we prove the following theorem. 

\begin{thm}\label{thm:csmcomplement}
Let $W_\A$ be the maximal wonderful compactification of 
the complement of an arrangement of hyperplanes $\A$ in
$\CC\PP^d$. 
Then 
$$\textstyle \bcsm(\mathbbm{1}_{{C}(\A)}) = \sum_{k=0}^d \csm_k(M_\A) \ \in \ \CH_*(W_\A) \cong \MW_{\ast}(\B_\A).$$
\end{thm}

The proof of Theorem \ref{thm:csmcomplement} relies on the next sequence of lemmas. 
For any chain of  flats  $\mathcal{F} \subseteq \hat \L_\A$ set $K_{\mathcal{F}} \coloneqq \bigcap_{F \in \mathcal{F}} D_F$
and  $K_{\mathcal{F}}^\circ \coloneqq K_{\mathcal{F}} \backslash \bigcup_{F \notin \mathcal{F}} D_F$.

\begin{lemma}\label{lem:Aluffi}
Let $W_\A$ be the maximal wonderful compactification of the complement of an arrangement of
hyperplanes $\A$ in $\CC\PP^d$.  
For any chain of  flats  $\mathcal{F} \subseteq \hat \L_\A$ 
we have 
$$\bcsm(\mathbbm{1}_{K_{\F}^\circ}) = \Big( c(T_{W_\A}(-\log(\mathcal{D})) 
\prod_{F \in \F} x_F \Big) \cap [W_\A] \in A_*(W_\A).$$
\end{lemma}

\begin{proof}
In the maximal wonderful compactification $W_\A$, the divisor $\mathcal{D} = W_\A \backslash {C}(\mathcal{A})$ 
is a simple normal crossing divisor, so  the lemma is a restatement of \cite[Lemma 5.4]{Aluffi}. 
\end{proof}

Before the next lemma we describe the operations of quotienting and  restricting hyperplane arrangements. 
Given an arrangement $\mathcal{A} = \{ H_0, \dots , H_n\}$ of hyperplanes in $\CP^d$, let $\hat{\mathcal{A}} = \{\hat{H}_0, \dots , \hat{H}_n\}$ 
denote the central hyperplane arrangement in $\C^{d+1}$ obtained by coning over the arrangement $\A$. 
Given a flat $F \in \mathcal{L}_{\mathcal{A}}$, let $\hat{H}_F = \bigcap_{i \in F} \hat{H}_i$. 
Then the quotient arrangement $\hat{\mathcal{A}} /  F$ is the central arrangement of hyperplanes in the vector space  $\C^{d+1} / \hat{H}_F$ given  by 
the collection $\{ \hat{H}_i / \hat{H}_F\}_{i \in F}$. 
The quotient arrangement $\mathcal{A} / F$ is the projectivization of this arrangement in $\mathbb{P}(\C^{d+1}/\hat{H}_{F})  \cong \CP^{\text{r}(F)-1}$.
The restriction arrangement  $\mathcal{A} | F$ is the arrangement of hyperplanes in 
$H_F \cong \CP^{d-\text{r}(F)}$ given by $\{H_i \cap H_F\}_{i \not \in F}$.

\begin{rem}\label{rem:matroidsQuotRestArrangements}
The matroid associated to the quotient arrangement $\mathcal{A} / F$ is the restriction $M_\A | F$, and the matroid associated to the restriction arrangement $\A | F$ is the contraction $M_\A / F$. 
This reverse correspondence is due to the fact that we are working with arrangements of hyperplanes in projective space instead of point configurations.
\end{rem}

\begin{lemma}\label{lem:prodComp}
Let $W_\A$ be the maximal wonderful compactifiation of the complement of an arrangement of
hyperplanes $\A$ in $\CC\PP^d$.  
For any chain of  flats $\mathcal{F} = \{ \emptyset  = F_0 \subsetneq F_1 \subsetneq \dots \subsetneq F_k \subsetneq F_{k+1} =  \{0, \dots, n\}\}$
of the matroid $M_{\A}$
we have 
$$K_{\mathcal{F}}^\circ \cong \prod_{i= 0}^k C( \mathcal{A}  / F_{i+1} | F_i ).$$
\end{lemma}

\begin{proof}

From \cite[Section 4.3]{deConciniProcesi}, the subvariety $K_{\mathcal{F}}$ of $W_{\A}$ is 
naturally isomorphic to the product 
$\prod_{i= 0}^k W_{\mathcal{A}_i}$, where $W_{\mathcal{A}_i}$ is the maximal wonderful compactification of
the complement of the arrangement $\A_i := \A / F_{i+1} | F_i$ in $\mathbb{P}(\hat{H}_{F_{i}} / \hat{H}_{F_{i+1}} )$. 

Firstly, consider the case when $\mathcal{F} = \{\emptyset \subsetneq F \subsetneq \{0, \dots, n\}\}$. 
Then $D_F = K_{\F} \cong W_{\mathcal{A}/ F} \times W_{\mathcal{A} |F}$.
For a flat $F'$ such that $F' \supsetneq  F$  we have that $K_{\F} \cap D_{F'} \cong W_{\mathcal{A} / F} \times E_{F'} $, where $E_{F'}$ is the proper transform of the subspace $H_{F'}$ of $H_{F}$ under the blow up of $H_{F}$ to the maximal wonderful compactification $W_{\mathcal{A} |F}$.
Similarly, for a flat  $F'$ such that $F' \subsetneq  F$ we have that $K_{\F} \cap D_{F'} \cong E_{F'} \times W_{\mathcal{A} |F}$, where in this case $E_{F'}$ is the divisor of  $W_{\mathcal{A} / F}$ corresponding to the 
proper transform of the subspace $\mathbb{P}(\hat{H}_{F'} /\hat{H}_F)$ of $\mathbb{P}(\C^{d+1}/\hat{H}_F)$. 
By removing all of these intersections we obtain $K^\circ_{\F} \cong C( \mathcal{A} / F ) 
\times C( \mathcal{A} | F)$.
 
The general claim now follows by induction on the length of the chain $\F$, in the same way as in the proof of the canonical isomorphism $K_\F \cong \prod_{i= 0}^k W_{\mathcal{A}_i}$ in \cite{deConciniProcesi}. 
\end{proof}

The following lemma is the key to relate CSM classes of complements of hyperplane arrangements
to CSM cycles of matroids.

\begin{lemma}\label{lem:csm0}
Let $\A$ be an essential hyperplane arrangement in $\CC\PP^d$.
 Then 
$$\Eu(C(\A)) = (-1)^d \beta(M_\A).$$ 
\end{lemma}

\begin{proof}
The Euler characteristic of $C(\A)$ is the evaluation of the reduced characteristic 
polynomial $\overline{\chi}_{M_\A}(\lambda)$ at $\lambda = 1$ \cite[Section 1.4.5]{Cohenetal},
which is equal to $(-1)^d \beta(M_\A)$ 
by Definition \ref{def:mobiusBeta}.
\end{proof}

\begin{lemma}\label{lem:eulerchar} 
Let $W_\A$ be the maximal wonderful compactification of the complement of an arrangement of
hyperplanes $\A$ in $\CC\PP^d$.  
For a chain of  flats $\mathcal{F} = \{ \emptyset  = F_0 \subsetneq F_1 \subsetneq \dots \subsetneq F_k \subsetneq F_{k+1} =  \{0, \dots, n\}\}$
of the matroid $M_{\A}$, we have
$$\Eu(K_{\mathcal{F}}^\circ)
= w_{\csm_k(M_\A)}(\sigma_{\mathcal{F}}).$$
\end{lemma}
\begin{proof}
By the behaviour of the Euler characteristic under Cartesian products and Lemma 
\ref{lem:prodComp}, we have
$\text{Eu}(K_{\mathcal{F}}^\circ) =  \prod_{i = 0}^k \text{Eu}(C(\A / F_{i+1} | F_i))$.
Recall that the matroid associated to  the arrangement $\A  / F_{i+1} |  F_i$ is $M_{\A} | F_{i+1} / F_i$. 
In view of Lemma \ref{lem:csm0}, we then have
$\text{Eu}(K_{\mathcal{F}}^\circ) = (-1)^{d-k}  \prod_{i=0}^{k} \beta(M_{\A} | F_{i+1}/F_i)$,
which agrees exactly with $w_{\csm_k(M_\A)}(\sigma_{\mathcal{F}})$ from Definition \ref{def:chernweight}. 
\end{proof}

\begin{rem}\label{rem:BB}
Bertrand and Bihan have a method for equipping skeleta of stable intersections of  tropical hypersurfaces with integer weights 
to produce balanced tropical cycles \cite{BertrandBihan}.  
In their construction, the weights are up to sign equal to the Euler characteristic of a non-degenerate complete intersection 
in a complex torus \cite[Theorem 5.9]{BertrandBihan}. 
The situation they consider overlaps with our own in the case of the stable intersection of fan tropical hyperplanes. 
These intersections give rise to the class of Bergman fans of cotransversal matroids. 
Lemma \ref{lem:eulerchar} shows that our weights  correspond to the same Euler characteristics, so that up to sign, the tropical CSM cycles we have defined coincide with the tropical cycles defined by Bertrand and Bihan for Bergman fans of cotransversal matroids.  
\end{rem}

\begin{proof}[Proof of Theorem \ref{thm:csmcomplement}]
Using the perfect pairing given in \eqref{eqn:compPairingExplicit},
it is enough to show that for any chain of flats 
$\F = \{\emptyset \subsetneq F_1 \subsetneq \dotsb \subsetneq F_k  \subsetneq  \{0, \dots, n\} \}$ we have 
$$x_{F_1} \dotsb x_{F_k} \cap \bcsm_k(\mathbbm{1}_{{C}(\mathcal{A})}) 
= x_{F_1} \dotsb x_{F_k} \cap \csm_k(M_\A).$$
The right hand side of the above equation is
$w_{\csm_k(M_\A)}(\sigma_{\mathcal{F}})$, 
which by Lemma \ref{lem:eulerchar} is equal to $\Eu(K_{\mathcal{F}}^\circ)$.
For the left hand side, first notice that 
$$\bcsm(\mathbbm{1}_{{C}(\mathcal{A})}) = c(T_{W_\A}(-\log\mathcal{D})) \cap [W_\A],$$
by setting $\F = \emptyset$ in Lemma \ref{lem:Aluffi}. 
Applying Lemmas \ref{lem:Aluffi} and \ref{lem:eulerchar} once again we obtain
\begin{align*}
x_{F_1} \dotsb x_{F_k} \cap \bcsm_k(\mathbbm{1}_{{C}(\mathcal{A})})& = 
  x_{F_1} \dotsb x_{F_k} \cap \Big( c_{d-k}(T_{W_\A}(-\log\mathcal{D}))  \cap [W_\A] \Big) \\
& = 
\Big( x_{F_1} \dotsb x_{F_k}c_{d-k}(T_{W_\A}(-\log\mathcal{D})) \Big) \cap [W_{\A}] 
\\
& = 
\bcsm_0(\mathbbm{1}_{K^\circ_{\mathcal{F}}})
\\
& = \Eu(K^\circ_{\mathcal{F}}),
\end{align*} 
completing the proof of the theorem. 
\end{proof}

\begin{rem}
There are other wonderful compactifications of the complement of a hyperplane arrangement $\A$
that generalize the maximal one considered here. These compactifications  arise from subsets of $\L_\A$ called \textbf{building sets} 
\cite{deConciniProcesi}. 
For any building set $\mathcal{G} \subseteq \mathcal{L}_ \A$ 
there is a {\bf nested set compactification} $W_{\mathcal{G}}$ of $C(\A)$, 
and the irreducible components of $\mathcal{D}_{\mathcal G} \coloneqq W_ \mathcal{G} 
\setminus C(\A)$ are in bijection with the elements of $\mathcal{G}$.
Given a building set $\mathcal{G} \subseteq \L_\A$ there is a birational map $f \colon W_{\mathcal{A}} \to W_{\mathcal{G}}$ 
consisting of a composition of blow downs of the exceptional divisors (and their pushforwards) corresponding to the flats in $\L_\A \backslash \mathcal{G}$. 
These blowdowns are \textbf{adaptive} in the sense of \cite[Lemma 1.3]{Aluffi:ChernOfBlowUps} and therefore we have $f_{\ast} \bcsm(\mathbbm{1}_{{C}(\A)})  =  \bcsm(\mathbbm{1}_{\tilde{{C}}(\A)})  \in \CH_{\ast}(W_{\mathcal{G}})$, where $\tilde{C}(\A)$ denotes the complement of $\A$ as a subset of $W_{\mathcal{G}}$. In this sense, the combinatorial  CSM cycle $\csm(M_{\A})$ defined here encodes the CSM class of the complement of the arrangement in any  wonderful compactification. 
\end{rem}

\section{CSM cycles are matroid valuations}\label{Sec:val}

In this section we prove that  
CSM cycles behave valuatively with respect to matroid polytope subdivisions.
We start with some general background on matroid polytopes and their subdivisions. 

To a collection  $\Gamma$ of subsets of 
$\{0,\dotsc,n\}$ we associate the polytope
$$\textstyle Q(\Gamma)\coloneqq \convex \{\mathbf e_S : S \in \Gamma \} \subseteq \RR^{n+1}, \quad \text{ where } \mathbf e_S\coloneqq \sum_{i \in S} \mathbf e_i.$$ 
Given a matroid $M\in\Mat_{n+1}$,
the \textbf{matroid polytope} $Q(M)$ is defined as the polytope 
$Q(\Bases(M)) \subseteq \RR^{n+1}$, where 
$\Bases(M)$ denotes  the collection of bases of $M$. 
The dimension of $Q(M)$ is equal to $n-c+1$, where $c$ is the number of connected components of $M$. 

Every face of a matroid polytope is again a matroid polytope, as we explain below.
If $Q \subseteq \RR^{n+1}$ is a polytope and $\mathbf v \in \RR^{n+1}$, we
denote by $\face_{\mathbf v}(Q)$ the face of $Q$ consisting of all $\mathbf x
\in Q$ maximizing the dot product with the vector $\mathbf v$, that is,
$$\face_{\mathbf v}(Q) \coloneqq \{ \mathbf x \in Q \mid \mathbf x \cdot \mathbf v \geq  \bfx'  \cdot \mathbf v  \text{ for all } \bfx' \in Q \}.$$ 
For any vector $\mathbf v \in \R^{n+1}$ there exists a unique cone of the form 
$\sigma_\F \coloneqq \cone(\bfe_{S_1}, \dotsc, \bfe_{S_{\ell}}) + \RR \!\cdot\! \bfe_{\{0,\dotsc, n\}}$ with $\F$ a flag of subsets 
$\{\emptyset = S_0 \subsetneq S_1 \subsetneq \dotsb \subsetneq S_{\ell} \subsetneq S_{\ell+1} = \{0,\dotsc,n\}\}$ such that $\mathbf v \in \interior(\sigma_\F)$.
In this case, the greedy algorithm for matroids implies that 
\begin{equation}\label{eq:face}
\face_{\mathbf v}(Q(M))  = Q(M|S_{1}/S_0 \oplus M|S_{2}/S_1 \oplus \dotsb \oplus M|S_{\ell+1}/S_{\ell}).
\end{equation}
Note that $\face_{\mathbf v}(Q(M))$ is the matroid polytope of a matroid with at least $\ell+1$ connected components, and so $\face_{\mathbf v}(Q(M))$ has dimension at most $n-\ell$.

Let $\Vertices(Q)$ denote the set of vertices of a polytope $Q$. 
A \textbf{subdivision} of a
$d$-dimensional  polytope $Q$ is a collection of 
$d$-dimensional  polytopes $\mathcal S=\{P_1, \ldots, P_m\}$
 such that for all $i$ we have  $\Vertices(P_i) \subseteq \Vertices(Q)$, $ Q = P_1 \cup \cdots \cup P_m $, and if an intersection $P_i \cap P_j$ is nonempty then it is a proper face of both $P_i$ and $P_j$. 
If all polytopes in a subdivision $\mathcal{S}$ of a matroid polytope $Q(M)$ are again matroid polytopes, then $\mathcal{S}$ is called a \textbf{matroid polytope subdivision}. 
 A {\bf face} of $\mathcal S$ is a face of any of the $P_i$, and the set of faces of $\mathcal S$ is denoted $\faces(\mathcal{S})$.  
 A  face of $\mathcal{S}$  is an  {\bf interior face} if it is not contained in the boundary of $Q$. The set of interior faces of $\mathcal S$ is denoted by $\interior (\mathcal S)$.

\begin{exa}
Let $n=3$, $d=1$, and consider the uniform matroid $M = U_{2,4}$. The matroid polytope $Q(M) \subseteq \RR^4$ is a regular octahedron, contained in the hyperplane $x_0 + x_1 + x_2 + x_3 = 2$. This matroid polytope admits three different non-trivial matroid subdivisions, each of which decomposes it into two square pyramids; see Figure \ref{fig:subdivision}.
\begin{figure}[ht]
\begin{center}
\includegraphics[scale=1.3]{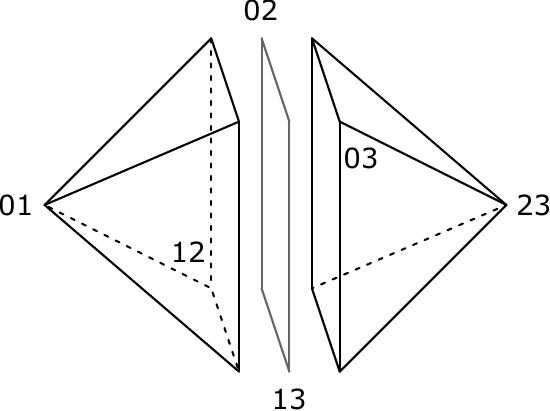}
\caption{A matroid subdivision of the polytope $Q(U_{2,4})$.}
\label{fig:subdivision}
\end{center}
\end{figure}
\end{exa}

\begin{defi}\label{def:valuation}
Let $\mathcal G$ be an arbitrary abelian group. A function $f \colon \Mat_{n+1} \rightarrow \mathcal G$ is a {\bf valuation under matroid polytope subdivisions}, or simply a {\bf valuation}\footnote{This use of the term \emph{valuation} in this way is standard in convex geometry. It should not be confused with the notion of a matroid valuation found in the theory of valuated matroids.}, 
 if for any matroid subdivision $\mathcal S$ of a matroid polytope $Q = Q(M)$ we have
\begin{equation*}\label{wellbeh}
\textstyle f(M) = \sum_{F \in \interior (\mathcal S)} \, (-1)^{\dim(Q)-\dim(F)} f(M_F),
\end{equation*}
where $M_F$ denotes the matroid whose matroid polytope is $F$.
\end{defi}

There is a slightly different definition of matroid valuations which captures more clearly the fact that $f$ can be computed by inclusion-exclusion on matroid polytopes \cite[Definition 3.1]{AFR}. 
This definition and the one given above are both equivalent by \cite[Theorem 3.5]{AFR}. 

\begin{exa}\label{ex:betavaluation}
The beta invariant and the characteristic polynomial from Definition \ref{def:mobiusBeta} are matroid valuations; see \cite{Speyer, AFR}. 
This implies that if $M \in \Mat_{n+1}$, for any subdivision $\mathcal S$ of $Q(M)$ we have
$$ \beta(M) = \sum_{\substack{F \in \interior  (\mathcal S) \\ \dim(F) = n}}  \beta(M_{F}),$$
since the beta invariant of any disconnected matroid is equal to zero. 
\end{exa}

\begin{exa}\label{p:valuation} For any $X \subseteq \R^{n+1}$, denote by $i_X \colon \Mat_{n+1} \to \ZZ$ the function assigning $1$ to a matroid $M$ if $Q(M) \cap X \neq \emptyset$, and $0$ otherwise. If $X$ is convex, and is either open or closed, it was shown in \cite[Proposition 4.5]{AFR} that $i_X$ is a matroid valuation.
\end{exa}

The following is the main result of this section. 
Recall that the set of fan tropical cycles $\mathcal{Z}_k(\R^{n+1} / \mathbf{1})$ forms a group 
under the operation of taking the union of supports and adding the weight functions.

\begin{thm}\label{thm:valuation}
For any $k$, the function $\csm_k \colon \Mat_{n+1} \to  \mathcal Z_k(\R^{n+1} / \mathbf{1})$ sending $M$ to $\csm_k(M)$ is a valuation under matroid polytope subdivisions.
\end{thm}

In order to prove Theorem \ref{thm:valuation} we need the following  lemmas. 

\begin{lemma}\label{l:cancellation}
Let $\mathcal S$ be a matroid polytope subdivision of a matroid polytope $Q \subseteq \RR^{n+1}$. 
For any fixed nonzero vector $\mathbf v \in \RR^{n+1}$ and 
face $F_0$ of $\mathcal S$, 
the function $j_{F_0, \mathbf{v}} \colon\faces(\mathcal S) \to \ZZ$
defined as 
\begin{equation}\label{eq:definitionj}
j_{F_0, \mathbf{v}}(F) \coloneqq
\begin{cases}
1 & \text{if } \face_{\mathbf v}(F) \supseteq F_0,\\
0 & \text{otherwise}
\end{cases}
\end{equation}
satisfies
 \begin{equation}\label{eq:sumlemma}
\sum_{F \in \interior (\mathcal S)} \, (-1)^{\dim(Q)-\dim(F)} \, j_{F_0, \mathbf{v}}(F) = 
\begin{cases}
1 & \text{if }  F_0 \subseteq \face_{\mathbf v}(Q),\\
0 & \text{otherwise.}
\end{cases}
\end{equation}
\end{lemma}

\begin{proof}
If the value of $\mathbf x \cdot \mathbf v$ is not constant when restricted to  all $\mathbf x \in F_0$, then 
$\face_{\mathbf v}(F) \nsupseteq F_0$ for any $F \in \faces(\mathcal S)$ and also for $F=Q$. In this case $j_{F_0, \mathbf{v}} \equiv 0$ and the statement is trivially true. 
Suppose $\mathbf v \cdot \mathbf x$ is constant for all $\mathbf x\in F_0$, and let $\mathbf x_0$ be a point in the relative interior of $F_0$. Consider the open half-space $H \colon= \{ \mathbf x \in \RR^{n+1} \mid \mathbf v \cdot \mathbf x>\mathbf v \cdot \mathbf x_0\}$. 
Let $C$ be a closed full-dimensional convex subset of $\overline{H}$ such that   $C \cap \partial \overline H= \mathbf x_0$ and the set $\Vertices(Q) \cap H$ is contained in the relative interior of $C$.
Let $i_C$ and $i_{\interior C}$ denote the matroid valuations discussed in Example \ref{p:valuation}.
We will show that $j_{F_0, \mathbf{v}} = i_C - i_{\interior C}$. 
The statement of the lemma will then follow directly from the fact that $i_C$ and $i_{\interior C}$ are both matroid valuations. 

Consider any face $F$ of the subdivision $\mathcal S$. Assume first that $F \subseteq \RR^{n+1} \backslash H$, so $i_{\interior C}(F) = 0$. In this case, $i_C(F) = 1$ if and only if $F \cap C = \mathbf \{\mathbf x_0\}$, which is equivalent to $F \supseteq F_0$ and thus to $\face_{\mathbf v}(F) \supseteq F_0$. It follows that $i_C(F) - i_{\interior C}(F) = j_{F_0, \mathbf{v}}(F)$, as claimed. 
If on the other hand $F \cap H \neq \emptyset$ then,  by the 
definition of $C$, we have $i_C(F) = i_{\interior C}(F) = 1$. Moreover, 
$\face_{\mathbf v}(F)$ is completely contained in a parallel translate 
of the hyperplane $\partial H$ lying inside $H$, and thus $j_{F_0, 
\mathbf{v}}(F) = 0$. Therefore $j_{F_0, \mathbf{v}}(F) = i_C(F) - 
i_{\interior C}(F)$ for all $F \in \faces(\mathcal{S})$, as desired. 
\end{proof}

If $\mathcal S_1$ is a subdivision of the polytope $Q_1 \subseteq \RR^{k_1}$ and $\mathcal S_2$ is a subdivision of the polytope $Q_2 \subseteq \RR^{k_2}$, the subdivision $\mathcal S_1 \times \mathcal S_2$ of $Q_1 \times Q_2 \subseteq \RR^{k_1+k_2}$ consists of all polytopes of the form $P_1 \times P_2$ with $P_1 \in \mathcal S_1$ and $P_2 \in \mathcal S_2$.

\begin{lemma}\label{l:product}
Let $Q_1 \subseteq \RR^{k_1}$ and $Q_2 \subseteq \RR^{k_2}$ be polytopes, and suppose $\mathcal S$ is a subdivision of the polytope $Q \coloneqq Q_1 \times Q_2 \subseteq \RR^{k_1 + k_2}$. If each  edge in $\faces(\mathcal S)$ is also an edge of $Q_1 \times Q_2$ then $\mathcal S = \mathcal S_1 \times \mathcal S_2$ with $\mathcal S_1$ a subdivision of $Q_1$ and $\mathcal S_2$ a subdivision of $Q_2$.
\end{lemma}
\begin{proof}
The faces of $Q$ have the form $F_1 \times F_2$ with $F_i$ a face of $Q_i$. In particular, all edges of $Q$ have the form $e_1 \times \{\bfw_2\}$ or $\{\bfw_1\} \times e_2$ with $e_i$ an edge of $Q_i$ and $\bfw_i$ a vertex of $Q_i$. We will refer to edges of the form $e_1 \times \{\bfw_2\}$ as ``vertical'' edges, and $\{\bfw_1\} \times e_2$ as ``horizontal'' edges. 
Fix a polytope $P \in \mathcal S$ and a vertex $(\bfv_1,\bfv_2) \in \Vertices(P) \subseteq  \Vertices( Q_1 \times Q_2)$. Let $V_1 \coloneqq \{ \bfw \in \Vertices(Q_1) \colon(\bfw,\bfv_2) \in P\}$ and $V_2 \coloneqq \{ \bfw \in \Vertices(Q_2) \colon(\bfv_1, \bfw) \in P\}$. 
We will show that $\Vertices(P) = V_1 \times V_2$, which implies the desired result. 

To prove the inclusion $\Vertices(P) \subseteq V_1 \times V_2$, consider any vertex $(\bfu_1, \bfu_2) \in \Vertices(P)$. 
By assumption, any edge of $P$ is also an edge of $Q$, so it is either vertical or horizontal. 
We claim that we can find a path from $(\bfv_1,\bfv_2)$ to $(\bfu_1,\bfu_2)$ in the  edge graph of $P$ which is a sequence of vertical edges followed by a sequence of horizontal edges.
The edge graph of $P$ is connected, so there exists a path $\gamma$ from  $(\bfv_1,\bfv_2)$ to $(\bfu_1,\bfu_2)$. 
Suppose that in $\gamma$ there is horizontal edge immediately followed by a vertical edge. Label the three vertices in this part of the path by $(\mathbf a_i, \mathbf b_i)$, 
$(\mathbf a_i, \mathbf b_{i+1})$ and $(\mathbf a_{i+1}, \mathbf b_{i+1})$. 
These three vertices are  in $\Vertices(P)$, and since $P$ cannot have the diagonal edge from 
$(\mathbf a_i, \mathbf b_i)$ to $(\mathbf a_{i+1}, \mathbf b_{i+1})$,
the polytope $P$ must also contain the vertex $(\mathbf a_{i+1}, \mathbf b_{i})$.
We can then alter $\gamma$ to pass by this vertex instead of $(\mathbf a_i, \mathbf b_{i+1})$, which replaces the horizontal edge followed by a vertical edge with a vertical edge followed by a horizontal one. 
Repeatedly applying this procedure we can produce the desired path. 
Now, once $\gamma$ has the desired form, the vertex of $P$ which connects the last  vertical edge of $\gamma$  to its first horizontal edge  must be $(\bfv_1,\bfu_2)$. Therefore,  $(\bfv_1,\bfu_2) \in \Vertices(P)$ and so $\bfu_2 \in V_2$. An analogous argument finds a path  from $(\mathbf u_1, \mathbf u_2)$ to $(\mathbf v_1, \mathbf v_2)$  consisting of horizontal edges followed by vertical edges. From this we can also conclude that  $\mathbf u_1 \in V_1$, as desired.

For the reverse inclusion $\Vertices(P) \supseteq V_1 \times V_2$, suppose $(\bfu_1, \bfu_2) \in V_1 \times V_2$. By the definition of $V_1$ and $V_2$ we have $(\bfu_1, \bfv_2), ( \bfv_1, \bfu_2) \in \Vertices(P)$. We can thus find a path $\gamma_1$ of edges in $P$ starting at $(\bfv_1,\bfv_2)$ and ending at $(\bfu_1,\bfv_2)$. Moreover, by the argument in the previous paragraph, we can assume that $\gamma_1$ consists of only vertical edges. Similarly, there is a path $\gamma_2$ in $P$ from $(\bfv_1,\bfv_2)$ to $(\bfv_1,\bfu_2)$ consisting of only horizontal edges. 
Suppose that $\gamma_1
 = (\bfv_1, \bfv_2) \to (\mathbf a_1, \bfv_2) \to \dotsb \to (\mathbf a_s, \bfv_2) \to (\bfu_1, \bfv_2)$ and $\gamma_2 = (\bfv_1, \bfv_2) \to (\bfv_1, \mathbf b_1) \to \dotsb \to (\bfv_1, \mathbf b_t) \to (\bfv_1, \bfu_2).$ 
Again, as $P$ cannot contain
any edges through the interior of a quadrangle of $Q$, applying an inductive argument to the successive quadrangles  formed by the two paths we arrive at the conclusion that every $(\mathbf a_i, \mathbf b_j)$ is in 
$\Vertices{P}$, and also $(\bfw_1, \bfw_2) \in \Vertices{P}$. This completes the proof of the lemma. 
\end{proof}

\begin{cor}\label{c:subdivision}
If $\mathcal S$ is a matroid polytope subdivision of the matroid polytope $Q(M_1 \oplus \dotsb \oplus M_k)$ then $\mathcal S = \mathcal S_1 \times \dotsb \times \mathcal S_k$ with $\mathcal S_i$ a subdivision of $Q(M_i)$ for all $i$. 
\end{cor}
\begin{proof}
Matroid polytopes can be characterized in terms of their edges: A polytope $Q$ with vertices in $\{0,1\}^{n+1}$ is a matroid polytope 
 if and only if all its edges are translations of vectors of the form $\mathbf e_i - \mathbf e_j$ for distinct $i,j \in \{0,1,\dotsc,n\}$ \cite{GGMS}.
Moreover, the edges of a matroid polytope $Q(M)$ are in correspondence with pairs of bases $A, B$ of $M$ satisfying $|A \setminus B| = |B \setminus A|= 1$.
This implies that if $\mathcal S$ is a matroid polytope subdivision of $Q$ then all the edges of $\mathcal S$ were already edges of $Q$.
The  statement of the corollary  now follows from Lemma \ref{l:product} and the fact that $Q(M_1 \oplus \dotsb \oplus M_k) \cong Q(M_1) \times \dotsb \times Q(M_k)$. 
\end{proof}

\begin{proof}[Proof of Theorem \ref{thm:valuation}]
We want to show that for any matroid $M \in\Mat_{n+1}$ and any matroid subdivision $\mathcal S$ of $Q = Q(M)$, 
$$\textstyle \csm_k(M)=\sum_{F\in \interior ( \mathcal S)}(-1)^{\dim(Q)-\dim(F)}\csm_k(M_F).$$
Denote by $\mathcal G$ the free abelian group
generated by the symbols $[F]$ with $F$ a face of $\mathcal S$. Consider the
homomorphism $\beta_k \colon\mathcal G \to \ZZ$ where $\beta_k([F])$ is equal to the
product of the beta invariants of all the connected components of $M_F$ if $M_F$
has exactly $k+1$ connected components, and 0 otherwise. In particular, if $\dim(F)\neq n-k$ then $\beta_k([F])=0$. 

Let $\F$ be a flag of subsets $\F \coloneqq \{\emptyset = S_0 \subsetneq S_1
\subsetneq \dotsb \subsetneq S_{k} \subsetneq S_{k+1} = \{0,\dotsc,n\}\}$, and denote
$\sigma_\F \coloneqq \cone(\bfe_{S_1},\dotsc,\bfe_{S_{k}}) + \RR \!\cdot\! \bfe_{\{0,\dotsc, n\}}$.
Fix a vector $\mathbf v$ in the relative interior of $\sigma_\F$. By Equation \ref{eq:face}, 
for any matroid $N \in \Mat_{n+1}$ we have
\[
\face_{\mathbf v}(Q(N)) = Q(N|S_{1}/S_0 \oplus N|S_{2}/S_1 \oplus \dotsb \oplus N|S_{k+1}/S_{k}).
\]
If $\rank(N) = d+1$, it follows that the weight of the cone $\sigma_\F$ in the cycle $\csm_k(N)$ is
\begin{equation*}\label{eq:csmbetak}
w_{\csm_k(N)}(\sigma_\F) = (-1)^{d-k} \beta_{k}([\face_{\mathbf v}(Q(N))]),
\end{equation*}
so it suffices to show that
\begin{equation}\label{eq:claimbetak}
\beta_k([\face_{\mathbf v}(Q)]) = \beta_k \bigg( \sum_{F \in \interior(\mathcal{S})} (-1)^{\dim (Q) -\dim (F)}[\face_{\mathbf v}(F)] \bigg).
\end{equation}

Starting from the right side of the above equation, we can write
$$\beta_k \Bigg( \sum_{F \in \interior (\mathcal S)}
(-1)^{\dim(Q)-\dim(F)} [\face_{\mathbf v}(F)] \Bigg) = \beta_k \Bigg( \sum_{\substack{G \in \faces(\mathcal S) \\ \dim(G) = n-k}}  a_{G} \, [G] \Bigg),$$
where $$a_{G} \coloneqq\sum_{\substack{F\in \interior (\mathcal S) \\ \face_{\mathbf v}(F) = G}}(-1)^{\dim(Q)-\dim(F)}.$$ 
For any face $F \in \interior(\mathcal S)$ we have $\dim(\face_{\mathbf{v}} (F)) \leq n-k$, 
as $\mathbf v$ is contained in the cone $\sigma_{\F}$ defined by a flag of $k$ non-trivial flats. 
Therefore, for $G$ of dimension $n-k$, the condition $\face_{\mathbf v}(F) = G$ is equivalent to $\face_{\mathbf v}(F) \supseteq G$. 
The coefficient $a_G$ is thus equal to
$$a_{G} =\sum_{F\in \interior (\mathcal S)}(-1)^{\dim(Q)-\dim(F)} j_{G, \mathbf{v}}(F),$$
where $j_{G, \mathbf{v}}$ denotes the function in Lemma \ref{l:cancellation}. 
By Lemma \ref{l:cancellation}, the coefficient $a_{G}$ is  equal to
$1$ if $G \subseteq \face_{\mathbf v}(Q)$ and $0$ otherwise. We have now  shown that 
\[
\beta_k \bigg( \sum_{F \in \interior(\mathcal S)} \, (-1)^{\dim(Q)-\dim(F)} \, [\face_{\mathbf v}(F)] \bigg) = \beta_k \bigg( \sum_{G \in \mathcal I} \, [G] \bigg),
\]
where $\mathcal I$ is the subset of $\faces(\mathcal S)$ consisting of the faces of dimension $n-k$ contained in $\face_{\mathbf v}(Q)$.

If $\dim(\face_{\mathbf v}(Q))<n-k$ then $\beta_k(\face_{\mathbf v}(Q)) = 0$ and also $\mathcal I = \emptyset$, 
so Equation \eqref{eq:claimbetak} is trivially true. Now assume that the dimension of $\face_{\mathbf v}(Q)$ is equal to $n-k$.
In this case, the faces in $\mathcal I$ are the top-dimensional polytopes in the   
subdivision $\mathcal S'$ of $\face_{\mathbf v}(Q)$ induced by $\mathcal{S}$. 
Since $\face_{\mathbf v}(Q) = Q(M|S_{1}/S_0 \oplus M|S_{2}/S_1 \oplus \dotsb \oplus M|S_{k+1}/S_{k})$, we can apply Corollary \ref{c:subdivision} to conclude that this  subdivision must have the form $\mathcal S' = \mathcal S'_1 \times \dotsb \times \mathcal S'_{k+1}$, where $\mathcal S'_i = \{ P^i_1,\dotsc,P^i_{m_i}\}$ is a subdivision of $Q(M|S_{i}/S_{i-1})$ for $1 \leq i \leq k+1$. 
Therefore, we have 
$$\beta_k \bigg( \sum_{G \in \mathcal I} \, [G] \bigg) =\sum_{J}  \prod_{i = 1}^{k+1}  \beta(M_{P^i_{J_j}}) = \prod_{i = 1}^{k+1} \sum_{j=1}^{m_i} \beta(M_{P^i_j}).$$ 
By Example \ref{ex:betavaluation} we have $\sum_{j=1}^{m_i} \beta(M_{P^i_j}) = \beta(M|S_{i}/S_{i-1})$ for all $i$, and so 
$$\beta_k \bigg( \sum_{G \in \mathcal I} \, [G] \bigg) = \prod_{i = 1}^{k+1}\beta(M|S_{i}/S_{i-1}) = \beta_k([\face_{\mathbf v}(Q)]),$$
which proves Equation \eqref{eq:claimbetak} and the statement of the theorem. 
\end{proof}

\section{Polynomial invariants from CSM cycles}\label{Sec:polynomial}

In this section we show how CSM cycles of matroids behave under deletions and contractions,  and we use this to express their degrees in terms of the coefficients
of the characteristic polynomial. We also provide a conjectural presentation of Speyer's $g$-polynomial of a matroid in terms of CSM cycles.

\subsection{Deletion and contraction of CSM cycles}

Recall that $\{\bfe_0, \bfe_1, \dotsc, \bfe_n\}$ denotes the standard basis of the lattice $\ZZ^{n+1} \subseteq \RR^{n+1}$. 
Fix $i \in \{0, \dots, n\}$, and let $\pi_i \colon \R^{n+1} \to \R^n$ denote the linear projection that forgets the $i$-th coordinate. 
With this projection in mind, we label the elements of the standard basis of $\Z^n \subseteq \R^n$ by 
$\bfe_k$ for $k \neq  i$. We will also denote by $\pi_i$ the induced map $\pi_i : \RR^{n+1}/\mathbf 1 \to \RR^{n}/\mathbf 1$. 

Let $M$ be a loopless matroid in $\Mat_{n+1}$. 
The flats of the deletion $M \backslash i$ and the contraction $M/i$ of $i$ from $M$ are 
\[
\L(M\setminus i) = \{ F \backslash i  \mid F \in \L(M) \} \quad \text{ and } \quad \L(M/i) = \{ F \backslash i  \mid F \in \L(M) \text{ and  } i \in F\};
\]
see, for example, \cite[Section 7]{White1}.
The map $\pi_i$ sends the cone $\sigma_ \F$ of the Bergman fan $\B(M)$ corresponding to a flag of flats
$\F: =\{ \emptyset \subsetneq F_1 \subsetneq \dots \subsetneq F_{k} \subsetneq \{0,\dotsc,n\}\}$ in $M$ to
the cone $\sigma_{\F'}$ where $\F'$ is the flag
$\F' \coloneqq\{ \emptyset \subseteq F_1\setminus i \subseteq \dots \subseteq F_{k}\setminus i \subseteq \{0,\dotsc,n\} \setminus i\}$.
It follows that the image of $\B(M)$ under $\pi_i$ is the Bergman fan $\B(M \backslash i)$. 
Let $\delta$ denote the restriction of $\pi_i$ to $\B(M)$. The surjective map $\delta \colon \B(M) \to \B(M\backslash i)$ is called the {\bf deletion map} with respect to the element $i$. 

The next proposition states that when $i$ is not a coloop of $M$ this deletion map is an {\bf open tropical modification} along a {\bf tropical rational function} $f \colon \R^{n} / \mathbf{1} \to \R$. We refer the reader to \cite{ShawInt} and \cite{BriefIntro} for an introduction to tropical modifications and tropical rational functions. 

\begin{prop}\cite[Proposition 2.25]{ShawInt}\label{prop:deletionMod}
Let $M \in \Mat_{n+1}$ be a loopless matroid and assume $i \in \{0, \dots, n\}$ is not a coloop of $M$. 
Then the deletion map 
$\delta \colon  \B(M) \to \B(M\backslash i)$ is an open tropical modification along a tropical rational  function $f: \R^{n} / \mathbf{1} \to \R$ 
such that 
$\divis_{\B(M\backslash i)}(f) = \B(M/i)$.  
\end{prop}

Proposition \ref{prop:deletionMod} is expressing the following fact.
If $i$ is not a coloop of $M$ then $M$ and $M \backslash i$ are matroids of the same rank, and thus their Bergman fans are of the same dimension. 
The map $\delta$ is one to one except above a codimension-$1$ subset of $\B(M \backslash i)$, which is exactly the Bergman fan $\B(M / i)$. 
The pre-image of $\delta$ over any point in $\B(M / i)$ is a half-line in direction ${\bf e}_i$.
The Bergman fan $\B(M)$ can be obtained from the graph of $f$ restricted to $\B(M \backslash i)$ by adding 
cones in the direction ${\bf e}_i$ over the
image of $\B(M/i)$. 

\begin{exa}
Consider the uniform matroid $M=U_{3,4}$ on the set $\{0, 1, 2, 3\}$. 
Then $M\backslash 3$ is the uniform matroid $U_{3,3}$ and $M / 3$ is the uniform matroid $U_{2,3}$. 
As we have seen in Example \ref{ex:coarse}, the Bergman fan $\B(M)$ is the union of the cones in $\R^4 / \mathbf{1}$ of the form $\cone\{\mathbf e_i, \mathbf e_j\}$ for all distinct  $i,j\in\{0,1,2,3\}$. The Bergman fan $\B(M\setminus 3)$ is all of $\R^3/ \mathbf{1}$, and $\B(M/3)$ is the union of the three rays in  $\R^3/ \mathbf{1}$ in the directions $\mathbf e_i$ for $i\in\{0,1,2\}$. 
Let $\pi_3\colon \R^{4} \to \R^{3}$ be the linear projection with kernel generated by ${\bf e}_3$.
This map induces the deletion map $\delta \colon \B(U_{3,4}) \to \B(U_{3,3})$, depicted in Figure \ref{fig:mod}. 
The tropical rational function $f: \R^{3}/ \mathbf{1} \to \R$  from Proposition \ref{prop:deletionMod} is in this case $f(\bfx_0, \bfx_1,\bfx_2)=\min\{\bfx_0, \bfx_1,\bfx_2\}$. 
\begin{figure}[ht]
\begin{center}
\includegraphics[scale=1.8]{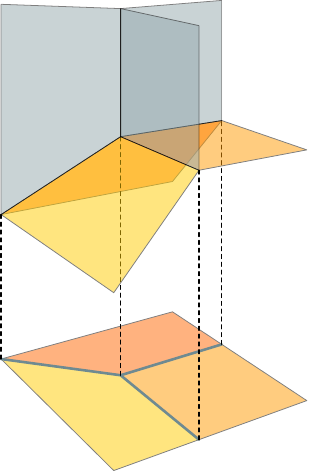}
\caption{The deletion map $\delta \colon \B(U_{3,4}) \to \B(U_{3,3}).$}
\label{fig:mod}
\end{center}
\end{figure}
\end{exa}

A deletion map between Bergman fans induces pushforward and pullback maps on tropical cycles.

\begin{defi}\cite[Definition 2.16]{ShawInt} \label{pushModCon}
Let $\delta \colon \B(M) \to  \B(M \backslash i)$ be the deletion map with respect to a non-coloop element $i$  of the loopless matroid $M$. 
For any $k$, the {\bf pushforward} and {\bf pullback maps} on tropical cycles are maps 
$$\delta_{\ast} \colon \mathcal{Z}_k(\B(M)) \to \mathcal{Z}_k(\B(M\backslash i)) \qquad \text{and} \qquad 
\delta^{\ast}  \colon \mathcal{Z}_k(\B(M\backslash i)) \to \mathcal{Z}_k(\B(M)).$$
The pushforward of a tropical cycle $Z \in  \mathcal{Z}_k(\B(M))$ is supported on the polyhedral complex $\delta(Z)$, and has weights described in \cite[Definition 2.16(1)]{ShawInt}. 
The pullback of a cycle $Z \in \mathcal{Z}_k(\B(M \backslash i))$ is 
the modification of $Z$ along the tropical polynomial function $f \colon \R^{n} / \mathbf{1} \to\R$
associated to $\delta$ by Proposition \ref{prop:deletionMod}.
\end{defi}

Both the pushforward and pullback maps induced by a deletion map $\delta \colon \B(M) \to \B(M \backslash i)$ are group homomorphisms. Moreover, the composition $\delta_{\ast}\delta^{\ast}$ is the identity in $\mathcal{Z}_k(\B(M\backslash i))$ \cite[Proposition 2.23]{ShawInt}.

We now use the pushforward and pullback homomorphisms to relate the CSM cycles of a matroid with the CSM cycles of its deletion and contraction with respect to a non-coloop element $i$.

\begin{prop}\label{prop:pullback}
Let   $\delta \colon \B(M) \to  \B(M \backslash i)$ be the deletion map with respect to a non-coloop element $i$ of the loopless matroid $M$. Then 
\begin{equation}\label{eq:pullback}
\csm_k(M) = \delta^*\csm_k(M\backslash i) -  \delta^*\csm_{k}(M/i)
\end{equation} 
and 
\begin{equation}\label{eq:pushforward}
\delta_{*}\csm_k(M) = \csm_k(M\backslash i) - \csm_{k}(M/i).
\end{equation}
\end{prop}

For the proof of Proposition \ref{prop:pullback} we need the following matroidal result, which we record as a separate lemma.
\begin{lemma}\label{lem:loopscoloops}
Let $S \subseteq T$ be subsets of the ground set of a matroid $M$, and suppose $i \in T \setminus S$.
\begin{enumerate}[label=\it{\alph*}\normalfont{)}]
\item \label{lem:coloop} If $T \setminus i$ is a flat of $M$ then $i$ is a coloop of $M|T/S$.
\item \label{lem:loop} If $S \cup i$ is a flat of $M$ but $S$ is not a flat of $M$ then $i$ is a loop of $M|T/S$.
\item \label{lem:notloopcoloop} If $S, T$ are flats of $M$ but $T \setminus i$ is not a flat of $M$ then $i$ is neither a loop nor a coloop of $M|T/S$.
\end{enumerate}
\end{lemma}
\begin{proof}
Recall that the circuits of the minor $M|T/S$ are the minimal nonempty subsets of the form $C \setminus S$,
where $C$ is a circuit of $M$ contained in $T$ \cite[Section 7]{White1}. This description implies that 
$i$ is a coloop of $M|T/S$ if and only if in $M$ the element $i$ is not in the closure of $T\setminus i$. Similarly, $i$ is a loop of
$M|T/S$ if and only if in $M$ the element $i$ is in the closure of $S$. The three assertions in the lemma follow directly from these facts. 
\end{proof}

\begin{proof}[Proof of Proposition \ref{prop:pullback}]
The second equation follows directly from the first one by applying $\delta_*$. 
To prove \eqref{eq:pullback}, suppose $\sigma_ \F$ is a $k$-dimensional cone of $\B(M)$ corresponding to the flag of flats
$\F: =\{ \emptyset = F_0 \subsetneq F_1 \subsetneq \dots \subsetneq F_{k} \subsetneq F_{k+1} = \{0,\dotsc,n\}\}$ in $M$.
The cone $\delta(\sigma_ \F)$ is the cone $\sigma_{\F'}$ where 
$\F' \coloneqq\{ F'_0 \subseteq F'_1 \subseteq \dots \subseteq F'_{k} \subseteq F'_{k+1}\}$
is the chain of flats of $M\setminus i$ defined by $F'_l \coloneqq F_l \setminus i$ for all $l$.

Assume first that $\sigma_{\F}$ is contained in the graph of the function $f \colon \R^{n} / \mathbf{1} \to \R$ restricted to $\B(M \backslash i)$, where $f$ is the tropical rational function of the modification $\delta$.
In this case $\sigma_{\F'}$ has the same dimension as $\sigma_ \F$, and so the chain $\F'$ has also length $k+1$.
By the pullback formula for tropical cycles, the weight of the cone $\sigma_ \F$ in $\delta^*(\csm_k(M\backslash i) - \csm_k(M/i))$ is equal to the weight of 
the cone $\sigma_{\F'}$ in the cycle $\csm_k(M\backslash i) - \csm_k(M/i)$.
To show that $\sigma_\F$ has the same weight in both cycles, we thus need to show that 
\begin{equation}\label{eq:equalweights}
\prod_{l=0}^k \beta(M|F_{l+1}/F_l) =   \prod_{l=0}^k \beta(( M \backslash i) |F'_{l+1}/F'_l)  + \prod_{l= 0}^k \beta(( M/ i) |F'_{l+1}/F'_l).
\end{equation}

Let $m$ be such that $i \notin F_m$ and $i \in F_{m+1}$. 
For any $l < m$, by Lemma \ref{lem:loopscoloops} \ref{lem:coloop} the element 
$i$ is a coloop in $M|(F_{l+1}\cup i)/F_l$, and thus its deletion is the same as its contraction, i.e., 
$M|F_{l+1}/F_l = ( M \backslash i) |F'_{l+1}/F'_l = (M/ i) |F'_{l+1}/F'_l$.
Moreover, since $\sigma_\F$ is in the graph of the function $f$, 
for any $l \geq m$ we have that $F_{l+1} \setminus i$ is not a flat of $M$,
otherwise the cone of $\B(M)$ corresponding to the chain of flats 
$\{ F'_0 \subseteq F'_1 \subseteq \dots \subseteq F'_{l+1} \subseteq F_{l+2} \subseteq \dotsb \subseteq F_{k+1}\}$ would be below the graph of $f$, contradicting Proposition \ref{prop:deletionMod}. 
Therefore, by Lemma \ref{lem:loopscoloops} \ref{lem:loop}, for any $l > m$ we have that 
$i$ is a loop in $M|F_{l+1}/(F_l \setminus i)$, and thus again $M|F_{l+1}/F_l = (M / i) |F'_{l+1}/F'_l = (M \setminus i) |F'_{l+1}/F'_l$.
When $l=m$, Lemma \ref{lem:loopscoloops} \ref{lem:notloopcoloop} shows that $i$ is neither a loop nor a coloop of $M|F_{m+1}/F_m$, and so we have
$$\beta(M|F_{m+1}/F_m) = \beta(( M \backslash i) |F'_{m+1}/F'_l) + \beta(( M/ i) |F'_{m+1}/F'_m).$$
Multiplying all these equations proves Equation \eqref{eq:equalweights}. This shows that the cycles $\csm_k(M)$ and $\delta^*(\csm_k(M \backslash i) - \csm_k(M/i))$ agree in the graph $\Gamma_f$ of the function $f$.

By the pullback formula for tropical cycles, any cone of the cycle $\delta^*(\csm_k(M \backslash i) - \csm_k(M/i))$ is either contained in $\Gamma_f$ or it contains the direction $\mathbf{e}_i$. Moreover, the weights of the cones contained in $\Gamma_f$, together with the balancing condition, determine the pullback cycle completely.
Similarly, each $k$-dimensional cone of the coarse subdivision of $|\B(M)|$ is either in $\Gamma_f$ or it contains the $\mathbf{e}_i$ direction. Since the support of the cycle $\csm_k(M)$ is the $k$-skeleton of this coarse subdivision (Proposition \ref{prop:nonzerocones}), the weights in $\csm_k(M)$ of the cones in the $\mathbf{e}_i$ direction are also determined by the weights of the cones in $\Gamma_f$ together with the balancing condition. This shows that the cycles $\csm_k(M)$ and $\delta^*(\csm_k(M \backslash i) - \csm_k(M/i))$ must be the same.
\end{proof}

\subsection{Degrees of CSM cycles and the characteristic polynomial}

We now relate the degrees of the CSM cycles of a matroid to the coefficients of its characteristic polynomial.
If $Z$ and $Z'$ are two tropical cycles 
in $\R^{n+1} / \mathbf{1}$, we denote by $Z \cdot Z'$ their stable intersection, and by $Z^k$ the stable intersection of $k$ copies of $Z$; see \cite[Section 3.6]{MaclaganSturmfels}.

\begin{defi}\label{def:degree}
The degree of a $0$-dimensional tropical cycle $Z$ in $\R^{n+1} / \mathbf{1}$ is $\deg(Z) \coloneqq \sum_{z \in Z} w_Z(z)$. 
The degree of a $k$-dimensional tropical cycle $Z$ in $\R^{n+1} / \mathbf{1}$ is 
$$\deg(Z) \coloneqq \deg( Z \cdot \B(U_{n, n+1})^k).$$
\end{defi}

\begin{exa}\label{ex:degcsmUniform}
Consider the uniform matroid $U_{d+1, n+1}$. 
By Example  \ref{ex:csmcycleUniform} we have $$\csm_k(U_{d+1, n+1}) = (-1)^{d-k}\binom{n-k-1}{d-k} \B(U_{k+1, n+1})$$ for all $0\leq k\leq d$. 
The degree of $\B(U_{k+1, n+1})$ is 1, and so $\deg(\csm_k(U_{d+1, n+1})) = (-1)^{d-k}\binom{n-k-1}{d-k}$.
\end{exa}

The following result generalizes \cite[Theorem 3.5]{Huh:MaxLikely} and \cite[Theorem 1.2]{Aluffi:Grothendieck} to all matroids, not necessarily representable in characteristic 0. Recall that $\overline{\chi}_M$ denotes the reduced characteristic polynomial of the matroid $M$.

\begin{thm}\label{thm:hvector}
If $M \in \Mat_{n+1}$ is a rank $d+1$ matroid then
$$\sum_{k=0}^d \deg(\csm_{k}(M))\,t^k = \overline{\chi}_M(1+t).$$
\end{thm}

\begin{exa}\label{ex:degcsm0d}
The $0$-dimensional CSM cycle of a rank $d+1$ matroid $M$ has degree equal to $(-1)^{d}\beta(M)$, which is equal to the constant coefficient $\overline{\chi}_M(1)$ of the polynomial $\overline{\chi}_M(1+t)$.
The $d$-dimensional CSM cycle of $M$ is equal to the tropical cycle $\B(M)$, which has degree $1$ if $M$ is loopless and $0$ otherwise. 
This is the leading coefficient of $\overline{\chi}_M(1+t)$. 
\end{exa}

We require the next proposition to prove  Theorem \ref{thm:hvector}.  

\begin{prop}\label{prop:degpullback}
Let $\delta : \B(M) \to \B(M\backslash i)$ be the deletion map with respect to a non-coloop element $i$ of $M$. 
For any $k$-dimensional tropical cycle $Z \in \mathcal{Z}_k(\B(M\backslash i))$,
we have $$\deg(Z) = \deg(\delta^*Z).$$ 
\end{prop}

\begin{proof}
To aid with notation we assume that $i = n$. 
The tropical cycle $\B(U_{n-1, n}) \in \mathcal{Z}_{n-2}(\R^{n}/ \mathbf{1})$ is the tropical hypersurface of the  tropical polynomial 
$h(\bfx_0, \dots, \bfx_{n-1}) = \min\{ \bfx_0, \bfx_1,  \dots, \bfx_{n-1}\}$ on $\R^{n}/\mathbf{1}$.
Let $C_n$ denote the matroid consisting of a single coloop $n$. Then $\B(U_{n-1, n} \oplus C_n) \in \mathcal{Z}_{n-1}(\R^{n+1} / \mathbf{1})$ is also a tropical hypersurface defined by the polynomial  $\tilde{h}(\bfx_0, \dots, \bfx_n) = h(\bfx_0  ,\dots, \bfx_{n-1})$. 
Let $\pi: \RR^{n+1} / \mathbf{1} \to \RR^{n}/ \mathbf{1}$ be the map induced by the  linear projection  $\RR^{n+1} \to \RR^{n}$ which forgets the $n$-th coordinate. 
 Then $\pi^{\ast} \divis(h) = \divis(\tilde{h})$, which implies that $\pi^{\ast} \B(U_{n-1, n}) = \B(U_{n-1, n} \oplus C_n)$. 

We have that $\pi_{\ast} \delta^{\ast}Z = \delta_{\ast} \delta^{\ast}Z = Z$.
Applying the projection formula in \cite[Proposition 4.8]{AllermannRau}
yields
\begin{align*}
\B(U_{n-1, n})^k \cdot Z &  =  \B(U_{n-1, n})^{k-1} \cdot (\divis(h) \cdot \pi_{\ast} \delta^{\ast}Z)  \\
 & =  \B(U_{n-1, n})^{k-1}  \cdot ( \pi_{\ast} (\pi^{\ast} \divis(h) \cdot \delta^{\ast}Z) ) \\ 
  & =  \B(U_{n-1, n})^{k-1}  \cdot \pi_{\ast} (\B(U_{n-1, n} \oplus C_n) \cdot \delta^{\ast} Z) .
\end{align*}
Repeatedly applying this argument $k$ times we obtain
$\B(U_{n-1, n})^k \cdot Z= \pi_{\ast}( \B(U_{n-1, n} \oplus C_n)^k \cdot \delta^{\ast}Z)$.
The degree of a zero cycle is preserved under the  pushforward map, and so we have
$\deg( Z) = \deg( \B(U_{n-1, n} \oplus C_n)^k \cdot \delta^{\ast}Z)$.

We will now show that $\deg( \delta^{\ast} Z) = \deg ( \B(U_{n-1, n} \oplus C_n)^k \cdot \delta^{\ast}Z )$.
Let $X \coloneqq \B(U_{n, n+1}) - \B(U_{n-1, n} \oplus C_n).$
Since $n$ is not a coloop of  $M$, the support of the tropical cycle $X$ is contained in the closed connected component of $\R^{n+1} / \mathbf{1}$ defined by 
$$\Gamma_f(\B(M\backslash n))^- \coloneqq \{ \mathbf x \in \R^{n+1} / \mathbf{1} \mid  \mathbf x \cdot \mathbf e_n \leq  f (\pi(\mathbf{x})) \}.$$ 
To compute the stable tropical intersection $\delta^{\ast}Z \cdot X$, denote by $X_{\epsilon}$ the  translate of $X$ by $ \epsilon \mathbf e_n$ for  $\epsilon < 0$. Then  $X_{\epsilon} \cap \delta^{\ast} Z = \emptyset$, and so 
$\delta^{\ast}Z \cdot X = 0$. Moreover, we have 
$$\delta^{\ast}Z \cdot  \big[\B(U_{n, n+1})^k - \B(U_{n-1, n} \oplus C_n)^k\big] 
= \delta^{\ast}Z \cdot  X \cdot  \left[  \sum_{j = 0}^{k-1}   
\B(U_{n, n+1})^{k-1-j} \cdot \B(U_{n-1, n} \oplus C_n)^j\right],$$
which is equal to zero by the associativity of the intersection product. 
This shows the equality of degrees $\deg( Z) = \deg(\delta^{\ast}Z)$ and proves the lemma. 
\end{proof}

\begin{proof}[Proof of Theorem \ref{thm:hvector}]
Both the reduced characteristic polynomial and the CSM cycles satisfy a recursion 
via deletions and contractions. More precisely, if $M$ is a loopless matroid and $i$ is not a coloop of $M$, we have
$$\overline{\chi}_M(\lambda) = \overline{\chi}_{M \backslash i}(\lambda) - \overline{\chi}_{M / i}(\lambda) \qquad \text{ and } \qquad 
\csm_k(M) = \delta^*(\csm_k(M \backslash i) -  \csm_{k}(M/i)),$$
where the equality on the right-hand side follows from Proposition \ref{prop:pullback}. 
Since degree is preserved under pullbacks by Proposition \ref{prop:degpullback}, in this case we have
\begin{equation}\label{eq:degreerecursion}
\deg(\csm_k(M)) = \deg(\csm_k(M \backslash i)) -  \deg(\csm_{k}(M/i)).
\end{equation}
If $M$ has any loops then
$$\overline{\chi}_M(\lambda) = 0 \qquad \text{ and } \qquad \csm_k(M) = \emptyset.$$
It thus suffices to check the statement for matroids $M$ with no loops and where all the elements are coloops, i.e., $M =  U_{d+1, d+1}$.
In this case, the tropical cycle $\B(M)$ is the same as $\R^{n+1} / \mathbf{1}$ equipped with weight $1$ everywhere. The only non-trivial CSM cycle is $\csm_d(\B(M)) = \B(M)$, which is of degree $1$.  Therefore  
$$\sum_{k=0}^n (-1)^k \text{deg}(\csm_{k}(M))t^k = t^d,$$ 
whereas by Example \ref{ex:betauniform},
$$\overline{\chi}_M(t+1) =\sum_{k = 0}^d (-1)^{d-k}\binom{d}{k}(t+1)^k= t^d,$$
confirming the desired result.
\end{proof}

\subsection{Conjecture: The $g$-polynomial as intersection numbers}

In this section we give a conjectured presentation of Speyer's $g$-polynomial of a matroid  using CSM cycles.

For a general rank $d$ matroid on $n$ elements, the $g$-polynomial of $M$ is defined by way of the $K$-theory of the Grassmannian $\text{Gr}(d,n)$ \cite{FinkSpeyer}. This polynomial is a valuative matroid invariant in the sense of Section \ref{Sec:val} \cite[Section 4]{FinkSpeyer}. 
Conjecture \ref{conj:gpoly} describes the coefficients of the $g$-polynomial as intersection numbers in the Bergman fan of $M$ between CSM cycles and certain tropical cycles defined recursively from them. 
This formula would offer a Chow theoretic description of this matroid invariant from $K$-theory.

There is an intersection product for tropical cycles contained in Bergman fans of matroids 
\cite{ShawInt, FrancoisRau}. 
If $M \in \Mat_{n+1}$ is a loopless rank $d+1$ matroid and $\mathcal Z_k(\B(M))$ denotes the group of $k$-dimensional tropical cycles whose support is contained in $\B(M)$, this intersection product gives rise to a 
bilinear pairing 
$$ \mathcal{Z}_{d-k}(\B(M))  \times \mathcal{Z}_{d-l}(\B(M)) \to \mathcal{Z}_{d-k-l}(\B(M)) $$
for any $k,l$ such that $k+l \leq d$.
In particular, for any $Z \in \mathcal{Z}_k(\B(M))$, the intersection product $\B(M) \cdot Z$ in the matroidal cycle $\B(M)$ is simply $Z$.

Using this product we define a
collection of new tropical cycles $n_k(M) \in \mathcal{Z}_k(\B(M))$ for $k = 0, \dots , d$. Firstly, we set 
$$n_d(M) \coloneqq \csm_d(M) = \B(M).$$
Let $A$ be the tropical cycle in $\mathcal{Z}_{d-1}(\B(M))$ obtained by taking the tropical stable intersection in $\R^{n+1} / \mathbf{1}$ of $\B(M)$ with the standard tropical hyperplane $\B(U_{n, n+1})$. 
For $k <d$ we define $n_k(M)$ recursively by the formula 
\begin{equation}\label{formula:Ncycles}
n_{d-k}(M) \coloneqq (-1)^{k} A^{k} - \Bigg [\sum_{i= 0}^{k-1} \csm_{d-k+i}(M) \cdot n_{d-i}(M) \Bigg ], 
\end{equation}
where the intersection products above are now  in $\B(M)$.

\begin{conj}\label{conj:gpoly}
The $g$-polynomial of a loopless rank $d+1$ matroid $M \in \Mat_{n+1}$ is equal to
\begin{equation}\label{eqn:conj}
g_M(t) = \sum_{k = 0}^d (-1)^{d-k}\,\deg(\csm_{k}(M)\cdot n_{d-k}(M))\,t^{k+1},
\end{equation}
where the intersection products occur in the matroidal cycle $\B(M)$ of $M$. 
\end{conj}

\begin{exa} For a loopless matroid  $M$ of rank $d+1$,  
Formula \eqref{formula:Ncycles} gives 
\begin{align*}
n_{d-1}(M) &= -A - \csm_{d-1}(M),\\
n_{d-2}(M) &= A^2 + A \cdot \csm_{d-1}(M) + \csm_{d-1}^2(M) - \csm_{d-2}(M).
\end{align*}
The linear, quadratic, and cubic coefficients of the 
polynomial on the right hand side of Equation \eqref{eqn:conj} are 
up to sign 
\begin{align*}
\csm_0(M) \cdot n_{d}(M) &= \csm_0(M) = (-1)^d \beta(M), \\
- \csm_1(M) \cdot n_{d-1}(M) &= \deg (\csm_1(M)) + \csm_1(M)\cdot\csm_{d-1}(M),\\
\csm_2(M) \cdot n_{d-2}(M) &= \deg (\csm_2(M)) + \deg (\csm_2(M) \cdot \csm_{d-1}(M)) \\
&\quad + \csm_2(M) \cdot \csm_{d-1}^2(M) - \csm_2(M) \cdot \csm_{d-2}(M).
\end{align*}

Consider the case $d=2$, so $M \in \Mat_{n+1}$ 
is a matroid of rank $3$ and $\B(M)$ is a $2$-dimensional tropical cycle. 
The intersection products above are 
\begin{align*}
\csm_0(M) \cdot n_2(M) &= \beta(M), \\
- \csm_1(M) \cdot n_{1}(M) &= \deg (\csm_1(M)) + \csm_1^2(M),\\
\csm_2(M) \cdot n_{0}(M) &= 1 + \deg (\csm_{1}(M)) + \csm_{1}^2(M) - \beta(M).
\end{align*}

For simplicity, let us assume that $M$ has no double points. By repeatedly applying Equation \eqref{eq:degreerecursion},  we find that $\deg(\csm_1(M)) = -(n-2)$. 
Moreover, 
the formula for intersection products of tropical cycles in $2$-dimensional Bergman fans in \cite[Definition 3.6]{BrugalleShaw} gives us 
$$\csm_1^2(M) = (n-2)^2 - \sum_{\substack{ F \in \mathcal{L}(M) \\ r(F) = 2}}(|F| -2)^2.$$
It can be verified that these formulae produce the coefficients of the $g$-polynomials in the examples of rank $3$ matroids presented in \cite[Section 10]{Speyer:Ktheory}. 
\end{exa}

\begin{exa}
Suppose $M$ is the uniform matroid $M=U_{d+1, n+1}$. 
In this case we have $A = \B(U_{d, n+1}) \in \mathcal{Z}_{d-1}(\B(M))$ and   
$A^k =  \B(U_{d-k+1, n+1}) \in \mathcal{Z}_{d-k}(\B(M))$. 
By Example \ref{ex:csmcycleUniform}, we have $\csm_k(M) = (-1)^{d-k}  \binom{n- k-1}{d-k} A^{d-k}$. 
 
We claim that $n_{d-k}(M) = \binom{n-d-1}{k}A^k$ for all $0 \leq k \leq d$. This formula 
is true when $k = 0$, so assume that it
holds for all $l < k$ and proceed by induction. 
By Formula (\ref{formula:Ncycles}) we have  
$$n_{d-k}(M)   = \left[(-1)^{k} - \sum_{i= 0}^{k-1} (-1)^{k-i}  \binom{n-d+k-i-1}{k-i} \binom{n-d-1}{i} \right]A^k.$$
Then the fact that $n_{d-k}(M) = \binom{n-d-1}{k}A^k$ follows from the binomial identity
$$(-1)^k = \sum_{i = 0}^k  (-1)^{k-i} \binom{m+k-i}{k-i}\binom{m}{i}$$
when $m = n-d-1$. 
From these expressions we conclude that 
$$\deg \left[ ( -1)^{d-k}\csm_{k}(M) \cdot n_{d-k}(M) \right] = \binom{n- k-1}{d-k}  \binom{n-d-1}{k}.$$
This coincides with the formula for the coefficients of the $g$-polynomial for uniform matroids \cite[Proposition 10.1]{Speyer:Ktheory}. 
\end{exa}

\bibliographystyle{amsalpha}
\bibliography{chernbib}

\end{document}